\documentclass[a4paper,10pt]{article}

\usepackage{amsmath,amsfonts,amsthm,amssymb}
\usepackage{enumerate}
\usepackage[top=2cm, bottom=2cm, left=1cm, right=1cm]{geometry}

\usepackage{bbm}

\newtheorem{theorem}{Theorem}
\newtheorem{definition}[theorem]{Definition}
\newtheorem{proposition}[theorem]{Proposition}
\newtheorem{corollary}[theorem]{Corollary}
\newtheorem{lemma}[theorem]{Lemma}
\newtheorem{remark}[theorem]{Remark}

\newtheorem{assumption}{Assumption A\hspace{-5pt}}

\numberwithin{equation}{section}

\newcommand{\RR}{\mathbb{R}}
\newcommand{\Sph}{\mathbb{S}}
\newcommand{\eps}{\varepsilon}
\newcommand{\Id}{\mathrm{Id}}
\newcommand{\bP}{\mathbf{P}}
\newcommand{\bC}{\mathbf{C}}

\newcommand{\abs}[1]{\left|#1\right|}
\newcommand{\R}{{\mathbb {R}}}
\newcommand{\norm}[2]{\left\| #1 \right\|_{#2}}
\newcommand{\nablax}{\nabla_{x}}

\title{Critical mass phenomenon for a chemotaxis kinetic model with spherically symmetric initial data}

\author{Nikolaos Bournaveas\protect\footnote{University of Edinburgh, School of Mathematics,
JCMB, King's Buildings, Edinburgh EH9 3JZ, UK.
On sabbatical leave to the Universit\'e Pierre et Marie Curie, Laboratoire Jacques-Louis-Lions, 175 rue du Chevaleret, 75013 
Paris, France. E-mail address:
\texttt{n.bournaveas@ed.ac.uk}}, Vincent Calvez\protect\footnote{\'Ecole Normale Sup\'erieure, D\'epartement de Math\'ematiques 
et Applications, 
45 rue d'Ulm, 75005, Paris, France. E-mail address: \texttt{vincent.calvez@ens.fr}}}

\date{\today}

\begin{document}

\maketitle

\begin{abstract}
The goal of this paper is to exhibit a critical mass phenomenon occuring in a model for cell self-organization {\em via} 
chemotaxis. The very well known dichotomy arising in the behavior of the macroscopic Keller-Segel system is derived at 
the kinetic level, being closer to microscopic features. Indeed, under the assumption of spherical symmetry,
 we  prove that  solutions with initial data of large 
mass blow-up in finite time, whereas solutions with initial data of small mass do  not. Blow-up is the consequence of a virial identity 
and the existence part is derived from a comparison argument. Spherical symmetry is crucial within the two approaches. We also 
briefly investigate the drift-diffusion limit of such a kinetic model. We recover partially at the limit the Keller-Segel 
criterion for blow-up, thus arguing in favour of a global link between the two models.
\end{abstract}

\section{Introduction}
In this paper we aim to exhibit a blow-up {\em versus} global existence phenomenon for a kinetic model describing collective 
motion of cells in two dimensions of space. The so-called Othmer-Dunbar-Alt  system \cite{ODA88,ABC} reads as follows,
\begin{equation}\label{genkinmodel}
\left\{ \begin{array}{l}  \displaystyle \partial_{t} f + v\cdot\nabla_{x} f = \int_{v'\in V} T[S](t,x,v,v') f(t,x,v')\, dv' - 
\lambda[S](t,x,v) f(t,x,v) \,,  \quad x\in \RR^2  \ ,\quad t>0\ ,   \medskip \\
 \displaystyle - \Delta S +\alpha S = \rho(t,x) =\int_{v\in V} f(t,x,v)\, dv\ , 
\end{array} \right.
\end{equation}
where $f(t,x,v)$ denotes the cellular density in position$\times$velocity space, and $S(t,x)$ is the concentration of the 
chemoattractant. The velocity set is assumed to be bounded, and for simplicity we take 
 $V = \mathcal{B}(0,R)$ throughout this paper. 
The initial data $f_0$ belongs to $L^1(\RR^2\times V)$ (more appropriate assumptions on $f_0$ will be stated within Assumption 
A\ref{ass:f0}, Theorems \ref{thm:BU} and \ref{thm:existence}). Observe that the mass of cells is formally conserved in time:
\[\iint_{\RR^2\times V}  f(t,x,v)\, dvdx = \iint_{\RR^2\times V} f_0(x,v)\, dv dx \, . \]

The turning kernel $T[S](t,x,v,v')\geq0$ denotes the probability of transition between  velocities $v'\to v$ at position $x$ and 
time $t$, and $\lambda[S] = \int_{v'\in V} T[S](t,x,v',v)\, dv'$ is the intensity of this Poisson process. The influence of the 
chemical field $S$ is highlighted in the notation $T[S]$. We make the particular choice \begin{equation} 
\label{eq:turning kernel}
T[S](t,x,v,v') = \chi_0 \big(v\cdot \nabla S(t,x)\big)_+\, ,
\end{equation} 
with constant $\chi_0>0$, that is to say cells choose only favorable directions when they reorientate, and they align more likely 
with the gradient of the chemical. This mechanism is actually not well-suited for describing accurately bacterial motion like the 
'run and tumble' process performed by {\em E. coli} \cite{ErbanOthmer06}. However, this fits well with motion of bigger and more 
complex cells capable of sensing a space gradient of the chemical and to orientate accordingly (amoeboid or 
mesenchymal motion \cite{ErbanOthmer07}). More complex kinetic models involving saturation effects or interactions between cells 
and the surrounding tissue have been investigated respectively in \cite{ChalubRodriguez,HillenM5}.

\subsection{Statement of the main results (blow-up vs. global existence)}

The kinetic model under consideration in this work can be written more precisely as follows:
\begin{equation}\label{kinmodel}
\left\{ \begin{array}{l} \displaystyle\partial_{t} f + v\cdot\nabla_{x} f = \chi_0(v\cdot \nabla S)_+ \rho  -\chi_0  \omega 
|\nabla S| f\, , \quad v\in V = \mathcal B(0,R)\,,  \quad x\in \RR^2\ ,\quad t>0 \, ,    \medskip \\
 \displaystyle - \Delta S + \alpha S  = \rho(t,x) \, . 
\end{array} \right.\, 
\end{equation}
where $ \omega =  \int_{v'\in V} \left(v'\cdot \nabla S/|\nabla S|\right)_+ dv' = 2 R^3/3$ (see Lemma \ref{lem:averaging} below, 
the precise value of $\omega$ does not play an important role in the sequel).

\begin{assumption}[Initial datum]
\label{ass:f0}
Assume that the initial density $f_0(x,v)\geq 0$ belongs to $L^1\cap L^\infty(\RR^2\times V)$. Assume in addition that $f_0$ has 
spherical symmetry: for any rotation $\Theta$, $f_0(\Theta x, \Theta v) = f_0(x,v)$. We denote by $M$ the total mass: \[M = \iint_{\RR^2\times V} f_0(x,v)\, dvdx\, .\]

\end{assumption}

This assumption ensures that the system \eqref{kinmodel} has a unique solution, which remains spherically symmetric 
during its life span (see Section \ref{localex} and Appendix). 

\begin{definition}[Blow-up in the context of kinetic chemotaxis]\label{def:BU}
A solution of \eqref{kinmodel} is said to blow-up if after some time (possibly infinite), $f(t,x,v)$ exits $L^p_xL^q_v$ for all exponents $p$ and $q$ such that 
\[
2< p \,, \quad 1< q  \, , \quad 0\leq\dfrac1q - \dfrac1p < \dfrac12\, .
 \]
This particular choice of exponents is clearly related to the dispersion Lemma \ref{lem:dispersion}. In fact $L^p_xL^q_v$ turns out to be a natural space for existence theory associated to \eqref{kinmodel} (see Appendix and \cite{BC08}).
\end{definition}

\begin{theorem}[Blow-up for large mass under spherical symmetry, $\alpha=0$] \label{thm:BU}
Assume that $f_0$ satisfies Assumption A\ref{ass:f0} and has a finite second moment w.r.t. the space variable $x$. Assume that 
the mass is large enough: 
\begin{equation} 
\label{eq:small mass BU}
M >  \dfrac{32\pi }{\chi_0 |V|} \, . \end{equation}
Then the solution of \eqref{kinmodel} with $\alpha =0$ blows-up in finite time. 
\end{theorem}

\begin{corollary}[Blow-up for large mass under spherical symmetry, $\alpha>0$]
\label{cor:BU}
Assume that $f_0$ satisfies Assumption A\ref{ass:f0} and has a finite second moment w.r.t. the space variable $x$. Assume that the mass is large enough \eqref{eq:small mass BU}.
Assume in addition that the following condition is fulfilled initially,
\begin{equation}
\label{eq:small I0}
\alpha \iint_{\RR^2\times V} |x|^2 \rho_0(x)\, dx < C(\chi_0,M,|V|)\, ,
\end{equation}
where $C(\chi_0,M,|V|)$ is an explicit (but heavy) constant vanishing when the case of equality is reached in \eqref{eq:small mass BU}.
Then the solution of \eqref{kinmodel} with $\alpha >0$ blows-up in finite time.  
\end{corollary}

\medskip

Before we state our next result we introduce some notation. 
For any exponent $0<\gamma<1$ define: \begin{equation}
\label{eq:Omega}
\Omega(\gamma) = 1 +  \dfrac1\pi \int_{\theta = -\pi/2}^{\pi/2}  \left|\sin \theta  \right|^{-\gamma} \,  d\theta\, . \end{equation}

\begin{theorem}[Global existence for small mass under spherical symmetry] \label{thm:existence}
Assume that $f_0$ satisfies Assumption A\ref{ass:f0} and lies  below $k_0 |x|^{-\gamma}$ for some $0<\gamma<1$ and $k_0>0$. Assume in addition that the mass is small enough: 
\begin{equation}
\label{eq:small mass} 
M \leq \dfrac{4\pi\gamma}{\chi_0 |V|\Omega(\gamma)} \, . \end{equation}
Then the solution of \eqref{kinmodel} exists globally in time.\\ In fact we can derive some pointwise estimate, and get that the solution does not blow-up in infinite time either. 
\end{theorem}

\begin{corollary}[Simplified criterion for global existence]
\label{cor:global existence}
Assume that $f_0$  satisfies Assumption A\ref{ass:f0} and additionally that $f_0\in C^0_c(\RR^2)$. There exists a threshold $M^*(\chi_0,|V|)$ such that if $M\leq M^*$ then the solution of \eqref{kinmodel} exists globally in time. 
\end{corollary}

\medskip

The two theorems above and their corollaries require a few comments.
\begin{enumerate}[(i)]
\item The meaning of spherically symmetric solutions in the context of kinetic models is given in the Appendix.
\item The case of the velocity space being $V = \Sph(0,R)$ is also investigated as a direct adaptation of Theorem \ref{thm:BU}
(see page \pageref{eq:integrate virial bis}).
 We obtain that blow-up occurs if the criterion \eqref{eq:small mass BU} is replaced with: 
 $M >  \left({16\pi }/{\chi_0 |V|}\right)$. In particular, cells' velocities are not likely to converge to zero in the general 
 case when blow-up occurs (both density and velocity collapse) and blow-up can happen even if the set of admissible velocities is 
 bounded from below. In this situation, the tumbling frequency dramatically increases in the neighbourhood of the blow-up 
 location.
\item Note that Theorem \ref{thm:BU} is included in its Corollary \ref{cor:BU} when $\alpha\to 0$. However, we state first Theorem \ref{thm:BU} for the sake of coherence (the Corollary can be thought of being a perturbation result for $\alpha>0$).
%
\item We shall derive in Section \ref{sec:BU} a weaker criterion for blow-up when $\alpha >0$, involving the initial first and second moments, and the initial current \eqref{eq:suff. cond.}. However we opted for a simplified presentation in Corollary \ref{cor:BU} involving the second moment only. 
\item Theorem \ref{thm:existence} and its corollary \ref{cor:global existence} are concerned with $\alpha=0$, but the case $\alpha>0$ can be proven similarly.
\item In Theorem \ref{thm:existence}, the assumption $f_0(x,v)\leq k_0 |x|^{-\gamma}$ is not very restrictive, because the tail $|x|^{-\gamma}$ is not integrable at infinity. 
\item We might have expected $0<\gamma<2$ instead of $0<\gamma<1$ because we are in dimension 2, and the natural space for solutions here is $L^1(\RR^2\times V)$. However it appears in \eqref{eq:Omega} that $\gamma<1$ is a crucial condition that we are not able to overcome (it underlies the fact that our reference function $k(x,v)$ we are comparing with in Section \ref{sec:existence} is not integrable with respect to velocity if $\gamma\geq 1$).
\item The continuous function $\gamma/{\Omega(\gamma)}$ goes to zero both for $\gamma\to 0$ and $\gamma\to 1$. Therefore there is a best compromise $\gamma^*$ which maximises the condition \eqref{eq:small mass}: $M\leq M^* = 2 \gamma^*/(\chi_0|V|\Omega(\gamma^*))$. However in this paper we keep general $\gamma$ for the sake of clarity.
Corollary \ref{cor:global existence} is nothing but applying Theorem \ref{thm:existence} to $\gamma=\gamma^*$ defined as above. 
\item Notice that the two critical mass thresholds (resp. \eqref{eq:small mass BU} and $M^*$) do not match, because $\Omega\geq2$ ensures $M^*\leq 2\pi \gamma^*/(\chi_0 |V|)< 2\pi /(\chi_0|V|)  $. Numerically we obtain $4 \gamma^*/\Omega(\gamma^*) \approx 0.806<32$!
\end{enumerate}

To the best of our knowledge, few blow-up results have been exhibited for kinetic models. Let us mention primarily the virial 
identities derived by Horst \cite{Horst82}, Glassey and Schaeffer \cite{GlasseySchaeffer} respectively for the Vlasov-Poisson and 
the relativistic Vlasov-Poisson models in the gravitational ({\em i.e} self-attracting) case (see also \cite{Glassey} for a 
presentation of these results). These remarkable identities allow the authors to show blow-up of the solutions having negative 
energy (resp. in dimension $d\geq4$, and in dimension $d=3$ under spherical symmetry). Recent progress aims to describe 
precisely the blow-up dynamics using the Hamiltonian structure and concentration compactness techniques for the Vlasov-Poisson 
system \cite{LemouMehatsRaphael}. Within the context of chemotaxis models, Chavanis and Sire have derived various virial theorems 
which share several common features with the identities derived in this paper \cite{ChavanisVirial}. 

Singularity formation plays a very important role in the kinetic theory of Bose-Einstein condensates, which arise when part of 
the particle density concentrates in the same quantum state \cite{EscobedoMischler,Saint-Raymond,Lu}. No virial identity has been 
developed in this theory however (in fact it would be irrelevant to argue through a vanishing second moment for proving concentration in this context): the convergence of the solution towards a singular limit (a Dirac mass together with a  regular part) at low temperature for the 
Boltzmann-Nordheim (resp. Boltzmann-Compton) equation is performed {\em via} entropy/entropy dissipation techniques (see \cite{EscobedoMischler,Lu} and \cite{Saint-Raymond} for an overview of the kinetic theory of Bose-Einstein condensates). Virial identities in the context of 
kinetic theory are also developed in cooling processes within the Boltzmann equation where particles are subject to inelastic 
collisions \cite{BobylevCarrilloGamba,MouhotMischler,BolleyCarrillo}. However, the context of the last two examples differs from 
the situation we are interested in  because concentration occurs in the `velocity' variable whereas for cell 
chemotaxis and self-attracting Vlasov-Poisson systems it occurs in the space variable. Consistently enough, the two aforementioned examples deal with homogeneous in space kinetic equations.

The paper is organized as follows: the next introductory subsections enlarge the picture, and replace the above results in a more general framework. We show the strong continuity between this critical mass phenomenon and previous existence theorems in kinetic theory for chemotaxis, and we explain the strong links between the current kinetic model, and the so-called parabolic Keller-Segel model. In Section \ref{localex} we briefly state a local existence theorem, which is later given with full details in Appendix. Section \ref{sec:BU} is devoted to the proof of the blow-up result having  different variants, using an efficient virial identity. In Section \ref{sec:existence} we prove global existence by a comparison argument with a very specific (and singular) reference function. The Appendix contains a short description of the meaning of spherically symmetric solutions in the context of kinetic equations, and also a general local existence Proposition  (based on dispersive estimates) to be crucially used in the comparison argument of Section \ref{sec:existence}.

\subsection{Brief review of context in the light of existence theory} 

As a by-product of this critical mass phenomenon we can argue that previous existence results in the kinetic theory of 
chemotaxis-biased cell motion were not far from being critical.
The state of the art of existence results is discussed in the introduction of \cite{BC08}. We recall below some of the 
issues for the sake of puting the results of our paper in context. 
From \eqref{genkinmodel} two classes of problems emerge depending on the assumption conditioning the turning kernel.
The same tool comes out to be powerful in both situations. As it will be used at few places in the present paper, it is worth 
recalling the dispersion lemma \cite{CastellaPerthame} which measures the action of the free transport operator on mild 
$L_x^pL_v^q$-norms:
\begin{lemma}(Dispersion estimate) \label{lem:dispersion} 
Let $g_{0}(x,v)\in L^{q}(\R^{2}  ; L^{p} (\R^{2}))$ where 
$1\leq q \leq p \leq \infty$, and let $g$ solve the free transport equation
\begin{equation}\label{kte}
\partial_{t} g + v \cdot \nabla_{x} g =0 \, ,
\end{equation}
with initial data $g(0,x,v)=g_{0}(x,v)$. Then 
\begin{equation}\label{dispest1}
\norm{g(t)}{L_x^{p}L_v^{q}} 
 \leq \frac{1}{t^{2\left(1/q-1/p\right)}}
\norm{g_{0}}{L_x^{q}L_v^{p}} .
\end{equation}
\end{lemma}
Strichartz estimates \cite{CastellaPerthame} have also been shown to apply succesfully to those run-and-tumble problems (see \cite{BCGP} and Remark \ref{rem:Strichartz} below).

\paragraph{Transport-dominating regime.} It deals with the case where the turning kernel can be estimated pointwise in terms of Sobolev norms of the chemical signal. For instance assumptions
\begin{equation} 
\label{eq:turning Sobolev} T[S](t,x,v,v') \leq C \|\nabla S(t)\|_\infty^{1-\nu}\, , \, 0< \nu\leq 1\, , \quad \mbox{or}\quad T[S](t,x,v,v') \leq C \|\nabla S(t)\|_r\, ,\, r<\infty\,  , \end{equation}
both lead to global existence of solutions (some superlinear power can in fact be added in the second case, see \cite{BC08} for details). When estimating the evolution of $L_x^pL_v^q$-norms, the dispersion due to the free transport operator turns out to have a strong enough effect to counterbalance the aggregation due to the tumble kernel.

\paragraph{Delocalization effects.} It deals with the case where the turning kernel is pointwise estimated through some space delocalization (volume effects, protrusion sending), for example,
\begin{equation} \label{eq:turning delocalized} T[S](t,x,v,v') \leq C |\nabla S(t,x+ \eps v)|\, , \quad \eps>0 .  \end{equation}
In this case the solution is again proven to be global in time \cite{HKS}. In \cite{BCGP} a second derivative subject to the same sort of delocalization can even be added to \eqref{eq:turning delocalized}.

It is worth noticing that our special choice of turning kernel \eqref{eq:turning kernel} is critical for both assumptions \eqref{eq:turning Sobolev} and \eqref{eq:turning delocalized}.

\begin{remark}[Dispersion method is borderline.]
The turning kernel we are studying in this paper satisfies
\begin{equation}\label{context1}
T[S](t,x,v,v') \leq C \norm{\nabla S(t)}{L^\infty}\, .
\end{equation}
It is natural to ask whether this property alone implies global existence. Indeed, if we work as
in \cite{BCGP,BC08} we arrive at
\[ \|\rho(t)\|_{L^p_xL^q_v}\leq 
\int_{0}^{t}\norm{\nabla S(s)}{L^\infty}\norm{\rho(s,x-(t-s)v)}{L^{p}_{x}L^{q}_{v}}\, ds\, .
\]
where $1\leq q \leq p \leq \infty.$
By Lemma \ref{lem:dispersion} the right-hand side can be controlled by
\[
C(V)\int_{0}^{t} \frac{1}{s^{2\left( 1/q - 1/p\right)}} 
\norm{\nabla S(s)}{L^\infty} \norm{\rho(s)}{L^{q}}\, ds\, .
\]
For $\rho$ we can use interpolation:
$\norm{\rho}{L^q} \leq \norm{\rho}{L^1}^{1-{p'}/{q'}}\norm{\rho}{L^p}^{{p'}/{q'}}$.
For $\nabla S$ we can use the elliptic estimate (see Lemma \ref{lem:elliptic}):
\[
\norm{\nabla S}{L^\infty} \leq C \norm{\rho}{L^1}^{1-{p'}/{2}} \norm{\rho}{L^p}^{{p'}/{2}}
\, , \,  2<p<\infty\, .
\]
We obtain eventually
\[
 \|\rho(t)\|_{L^p_xL^q_v}\leq C \int_{0}^{t} \frac{1}{s^{2\left( 1/q - 1/p\right)}} \norm{\rho(t-s)}{L^p}^{{p'}/{2} + {p'}/{q'} }\,  ds\, ,
\]
and it is impossible to achieve both $2\left( 1/q - 1/p\right)<1$ for integrability near $s=0$ and 
$ {p'}/{2} + {p'}/{q'} \leq 1$ for applying a global Gronwall's lemma.
\end{remark}

\begin{remark}[Strichartz method is borderline.]
 \label{rem:Strichartz}
If we are willing to impose a smallness condition we can try to use Strichartz estimates in the same spirit as \cite{BCGP}. Recall from \cite{CastellaPerthame} that if $f$ solves the free transport equation
\[
\partial_t f + v \cdot \nabla_x f =g\, ,
\]
then 
\begin{equation}\label{context2}
 \norm{f}{L^{q}_{t}L^{p}_{x}L^{r}_{v}}\leq C_0 + C_1 \norm{g}{L^{q'}_{t}L^{r}_{x}L^{p}_{v}}\, .
\end{equation}
where $C_0$ depends only on the initial data and the parameters $q,p,r$ satisfy
\begin{equation}\label{context3}
1\leq r \leq p \leq \infty,\ \ \ \frac{2}{q}=2\left(\frac{1}{r}-\frac{1}{p}\right)<1, \ \ \ 
\frac{1}{p} + \frac{1}{r}\geq 1\, .
\end{equation}
The borderline case (when the middle condition in \eqref{context3} fails) would correspond to $q=2,\, p=2,\, r=1$,
\begin{equation}\label{context4}
 \norm{f}{L^{2}_{t}L^{2}_{x}L^{1}_{v}}\leq C_0 + C_1 \norm{g}{L^{2}_{t}L^{1}_{x}L^{2}_{v}}\, .
\end{equation}
We are interested in $g= C \norm{\nabla S}{L^\infty}\rho$ under the assumption of spherical symmetry.
In this special case we have 
\[
\forall x\quad \abs{\nabla S(x)}=\frac{1}{r}\int_{0}^{r}\rho(\lambda) \lambda\ d\lambda \leq 
\frac{1}{r} \left(\int_{0}^{r} \lambda\ d\lambda \right)^{1/2} 
\left(\int_{0}^{r}\rho(\lambda)^2 \lambda\ d\lambda \right)^{1/2} 
\leq C \norm{\rho}{L^2}\, ,
\]
therefore
\begin{equation}
 \norm{f}{L^{2}_{t}L^{2}_{x}L^{1}_{v}}\leq C_0 + C_1 \big\| \norm{\rho(t)}{L^2}\rho(t,x)\big\|_{L^{2}_{t}L^{1}_{x}}
= C_0 + C_1 M \norm{f}{L^{2}_{t}L^{2}_{x}L^{1}_{v}}.
\end{equation}
If the mass $M$ was small enough we would be able to bootstrap.

Thus an alternative proof of global existence under small mass and spherical symmetry (much simpler than the one we develop in Section \ref{sec:existence}) would rely on a critical Strichartz estimate that we are not currently able to handle. 
\end{remark}

\subsection{A reminder of the classical Keller-Segel in 2D of space}

The critical mass phenomenon studied in this paper shares several similarities with the qualitative 
behaviour of the parabolic Keller-Segel system in two dimensions of space:
\begin{equation}
\left\{
\begin{array}{l}
 \displaystyle \partial_t \rho  = \Delta \rho - \chi_0\nabla\cdot \left(\rho\nabla S\right)\, , \quad t>0\, , \, x\in \RR^2\, , \medskip \\
-\Delta S + \alpha S = \rho\, .
\end{array}
\right. 
\end{equation}
In fact, there is a simple dichotomy: if the mass is below the threshold $M<8\pi/\chi$ then the solution is global in time and disperses with the space/time scaling of the linear heat equation; on the other hand, if it is above the same threshold $M>8\pi/\chi$, then the solution blows-up in finite time (in the case $\alpha = 0$). For blow-up in the case $\alpha>0$ one usually adds an hypothesis close to \eqref{eq:small I0} \cite{CalvezCorrias}. This critical mass phenomenon was first derived in a bounded domain with radial symmetry \cite{JL92,Nagai95}. Energy methods based on {\em ad-hoc} functional inequalities (either Trudinger-Moser or Hardy-Littlewood-Sobolev with a logarithmic kernel) were developed later on \cite{GZ98,BDP06}.

The analogy is not complete however, as can be seen in the details of our proofs.  Concerning the blow-up, we have to differentiate twice in time the virial identity as opposed to  Keller-Segel for which it holds true (when $\alpha = 0$):
\[ \dfrac d{dt}  \frac12 \int_{\RR^2}|x|^2 \rho(t,x)\, dx = 2 M\left( 1- \dfrac{\chi M}{8\pi}\right)\, . \]
Concerning global existence, the Keller-Segel system is equipped with a free energy (entropy minus chemical potential energy) which is dissipated along the trajectories and this provides useful {\em a priori} estimates ensuring global existence for small mass. 
No such energy is known at the kinetic level.
In the present work we use a comparison principle with a singular but integrable reference function.

\subsection{Drift-diffusion limit (formal)}

The parabolic Keller-Segel system can be obtained as a drift-diffusion limit of the kinetic Othmer-Dunbar-Alt model \cite{OH02,CMPS04,CDMOSS06}, when the chemotaxis bias is  a small perturbation of an unbiased process. We may express this fact by modifying the turning kernel under consideration:
\begin{equation} 
\label{eq:turning kernel bis}
T_\epsilon[S](t,x,v,v') = F(v) + \epsilon \chi_0 \left(v\cdot \nabla S(t,x) \right)_+ \, ,
\end{equation} 
instead of \eqref{eq:turning kernel}.  Thus, the jump process consists in the superposition of a relaxation process (towards a velocity distribution $F(v)\geq0$ such that $\int vF(v)dv = 0$ and $\int  F(v) dv = 1$) and a small bias due to chemotaxis. We assume $F(v)$ to be rotationally symmetric (in order to match with the context of this paper).

\begin{remark}
Previous works (see {\em e.g.} \cite{FilbetLaurencotPerthame}) state in general that the scattering operator $\mathcal T_\epsilon[f,S](x,v) = \int_V T_\eps[S] f(v')\, dv' - \lambda_\epsilon[S] f(v)$ can be decomposed as $\mathcal T_0[f] + \epsilon \mathcal T_1[f,S]$, where the linear unbiased operator $\mathcal T_0[f]$ possesses an equilibrium configuration with respect to velocity, namely there exists a probability distribution $F(v)$ such that $\int_V vF(v) dv = 0$ and $\mathcal T_0[F] = 0$. In this paper however we restrict the presentation to the special case of relaxation towards $F$ for the sake of clarity, but the standard procedure can be performed in the same way.
\end{remark}

The kinetic model with the parabolic scaling writes
\begin{equation}\label{eq:kinmodel bis}
\left\{ \begin{array}{l}  \epsilon \displaystyle\partial_{t} f_\epsilon + v\cdot\nabla_{x} f_\epsilon = \dfrac1\epsilon   \left( \rho_\epsilon F(v) - f_\epsilon  +\epsilon \chi_0\left(v\cdot  \nabla S_\epsilon \right)_+ \rho_\epsilon  -\epsilon \chi_0  \omega |\nabla S_\epsilon| f_\epsilon \right)\, , \medskip \\
 \displaystyle - \Delta S_\epsilon + \alpha S_\epsilon  = \rho_\epsilon(t,x) \, . 
\end{array} \right.\, 
\end{equation}
Formally, as $\epsilon\to 0$, the cell density $f_\epsilon$ decouples into a product $\rho(t,x) F(v)$ (so that the leading order term cancels), and $\rho$ is to be determined. To do so, integrate against $1$ and $v$ the first line of \eqref{eq:kinmodel bis} and get respectively
\begin{align*}
&\epsilon \partial_t \rho_\epsilon + \nabla\cdot j_\epsilon = 0\, , \quad \mbox{with}\quad j_\epsilon = \int_V v f_\epsilon\, dv \, ,\\
&\epsilon \partial_t j_\epsilon + \nabla\cdot \left(\int_V v\otimes v f_\epsilon\, dv\right) = \frac1\epsilon \left( - j_\epsilon + \epsilon \chi_0 \dfrac{|V|^2}{8\pi}  \left(\nabla S_\epsilon\right)  \rho_\epsilon - \epsilon \chi_0 \omega |\nabla S_\epsilon| j_\epsilon \right)\, .
\end{align*}
Still formally, we obtain the renormalized flux for small $\epsilon$
\[ \dfrac{j_\epsilon}\epsilon  = -   \nabla\cdot\left( \int_V v\otimes v f_\epsilon \, dv\right) + \chi_0 \dfrac{|V|^2}{8\pi}  \left(\nabla S_\epsilon\right) \rho_\epsilon\, . \]
Therefore we obtain as $\epsilon$ goes to zero, the limiting parabolic equation for the cell density in space $\rho(t,x)$,
\begin{equation}
\label{eq:parabolic limit}
\partial_t\rho  = \nabla\cdot \left(  \left[\int_V v\otimes vF(v)\, dv\right] \nabla \rho \right) -  \dfrac{\chi_0|V|^2}{8\pi} \nabla\cdot\left(  \rho \nabla S    \right) \, ,
\end{equation}
coupled with the chemical potential equation $-\Delta S + \alpha S = \rho$.

Notice that the current assumptions fit with the framework of \cite{CMPS04}, which makes this analysis rigorous for short time $t<t^*$ (independent of $\epsilon$). We refer to the end of Section \ref{ddl} for a discussion about this formal limit from the viewpoint of blow-up results. In a short, we show that blow-up criterions are indeed the same (asymptotically) for the kinetic model \eqref{kinmodel} and its parabolic limit.
This raises the question whether this convergence is still valid for larger times. In other words: do the kinetic and the parabolic Keller-Segel systems remain close to each other throughout their respective periods of existence? Our result strongly supports the fact that the solutions are indeed close for all time (before the blow-up time). However, a rigorous statement together with a full proof of convergence for all time has yet to be performed.


\section{Preliminaries: local in time existence and uniqueness}\label{localex}

In this Section we prove local existence and uniqueness (without the assumption of spherical symmetry) for the system \eqref{kinmodel}. 
\begin{proposition}\label{localexthm}
 Consider the model \eqref{kinmodel} with the turning kernel given by \eqref{eq:turning kernel}. Fix $p\in (2,\infty)$
and suppose that $f_0 \in L^{1}_{x,v}\cap L^{p}_{x,v}$. Then there exists a positive
number $T$ depending only on $f_0$ and a  unique solution $f$ with 
\[f \in L^{\infty}\left([0,T] ; L^{1} \left(\R^2 \times V\right) \right)\cap 
 L^{\infty}\left([0,T] ; L^{p}\left(\R^2 \times V\right) \right)\, \]
\end{proposition}

Before going into the proof of Proposition \ref{localexthm}, let us state a useful elliptic estimate.
We omit the easy proof which is a direct consequence of the Hardy-Littlewood-Sobolev inequality.
\begin{lemma}[Elliptic estimate] \label{lem:elliptic}
Let $p>2$ and $S$ be the solution of $-\Delta S = \rho$ in the sense
\[\nabla S(x) = -\frac{1}{2\pi}\int \frac{x-y}{|x-y|^2} \rho(y)\, dy\, .\]
Then 
\[ \|\nabla S\|_\infty \leq C(p)  \norm{\rho}{L^{1}}^{1-\frac{p'}{2}} 
\norm{\rho}{L^{p}}^{\frac{p'}{2}} \, , \quad \lim_{p\to 2^+}C(p) = +\infty\, . \]
\end{lemma}
Note that this elliptic estimate holds true for $p=2$ in the spherically symmetric case. However, we shall not use that variant here.

\begin{proof}[Proof of Proposition \ref{localexthm}]
Let us write the nonlinear kinetic equation of interest as
\begin{equation*}\label{localex1}
 \partial_t f + v \cdot \nabla_x f = N(f)\, ,
\end{equation*}
where the nonlinear scattering operator is given by 
\begin{equation*}\label{localex2}
 N(f)=\chi_0 \left(v \cdot \nabla S\right)_{+} \rho - 
\chi_0 \omega \abs{\nabla S}f  \, .
\end{equation*}
We shall prove that the nonlinear operator satisfies a Lipschitz estimate,
\begin{equation}\label{lip}
 \norm{N(f_1) - N(f_2)}{X}\leq L\left(\norm{f_1}{X}, \norm{f_2}{X}\right)\norm{f_1 - f_2}{X}\, ,
\end{equation}
where the norm is defined by
\begin{equation}\label{norm}
  \norm{g}{X}  =\sup_{0\leq t \leq T} \left(\norm{g(t,x,v)}{L^{1}_{x,v}} + 
 \norm{g(t,x,v)}{L^{p}_{x,v}}\right)\, .
\end{equation}
We split the difference of the nonlinear contributions into four different parts, namely, 
\begin{equation*}
 \begin{array}{rll}
 N(f_1) - N(f_2) = & \chi_0 
\left(\left(v \cdot \nabla S_1\right)_{+}
- \left(v \cdot \nabla S_2\right)_{+}\right)\rho_1  & \quad (I) \smallskip\\
&  + \chi_0 \left(v \cdot \nabla S_2\right)_{+} \left( \rho_1 - \rho_2\right)  & \quad (II) \smallskip\\
&  - \chi_0\omega\left(\abs{\nabla S_1} - \abs{\nabla S_2} \right)f_1  & \quad (III) \smallskip\\
& - \chi_0\omega \abs{\nabla S_2}\left( f_1 - f_2\right) &\quad (IV) \, .
\end{array}
\end{equation*}
For the first contribution $I$ we have
\begin{align*}
 \abs{I(t,x,v)}
&\leq C(\chi_0, V) \norm{\nabla S_1(t) -\nabla S_2(t) }{L^\infty} |\rho_1(t,x)| \, .
\end{align*}
The difference of the two gradients in $L^\infty$ can be estimated via the elliptic estimate of Lemma \ref{lem:elliptic}:
\begin{align*}
\norm{\nabla S_1(t) -\nabla S_2(t) }{L^\infty}&\leq C(p)  \norm{\rho_1(t) -\rho_2(t) }{L^{1}}^{1-\frac{p'}{2}} 
\norm{\rho_1(t) -\rho_2(t) }{L^{p}}^{\frac{p'}{2}} \, ,  \\
&\leq C(p,V) \left(\sup_{0\leq t' \leq T} \norm{f_1(t') - f_2(t') }{L^{1}_{x,v}}^{1-\frac{p'}{2}}\right)
\left(\sup_{0\leq t' \leq T}\norm{f_1(t') - f_2(t') }{L^{p}_{x,v}}^{\frac{p'}{2}}\right)\nonumber\\
&\leq C(p,V) \norm{f_1 - f_2}{X}\, .  
\end{align*}
As a consequence we get a Lipschitz condition for the first part $I$:
\begin{eqnarray*} 
 \abs{I(t,x,v)} &\leq& C(p,\chi_0, V) \norm{f_1 - f_2}{X}  |\rho_1 (t,x)| \, ,\\
 \norm{I }{L^{p}_{x,v}} &\leq& C(p,\chi_0, V) \norm{f_1 - f_2}{X}  \norm{\rho_1 }{L^{p}}\\
&\leq &C(p,\chi_0, V) \norm{f_1 - f_2}{X} \norm{f_1}{X} \, .
\end{eqnarray*}
Similarly 
\begin{equation*}\label{localex41}
 \norm{I }{L^{1}_{x,v}} \leq C(p,\chi_0, V)  \norm{f_1 - f_2}{X} \norm{f_1}{X}\, ,
\end{equation*}
therefore
\begin{equation}\label{localex42}
 \norm{I}{X} \leq C(p,\chi_0, V) \norm{f_1}{X} \norm{f_1 - f_2}{X} .
\end{equation}
The estimates for $II$, $III$ and $IV$ are obtained analogously aso that  we end-up with the desired estimate \eqref{lip}
with $L\left(\norm{f_1}{X}, \norm{f_2}{X} \right)=C(p,\chi_0, V) \left(\norm{f_1}{X} + \norm{f_2}{X}\right)$.

To conclude, let us mention that the norm $\|\cdot\|_X$ defined by \eqref{norm} is preserved through the action of the free 
transport operator, thus a fixed-point argument can be developped and leads to the conclusion.
\end{proof}

\begin{remark}
In the appendix we state a more complex existence/uniqueness result in suitable spaces $L^p_xL^q_v$. This is to fit with the comparison method of Section \ref{sec:existence}.
\end{remark}

\section{Formation of a singularity for large mass}

\label{sec:BU}

\subsection{A blow-up criterion in the case $\alpha = 0$}

We first need to state a technical Lemma for explicit computations.

\begin{lemma}[Averaged quantities] \label{lem:averaging}
Recall that $V = \mathcal B(0,R)$. 
\begin{enumerate}[(i)]
\item For any $q\in \RR^2$ we have \label{item:i}
\[ \int_{V} (v\cdot q)_+\, dv = \dfrac{2R^3}3 |q|\, . \]
\item For any $(p,q)\in \RR^2\times \RR^2$, we have \label{item:ii}
\[ \int_{V} (p\cdot v)(v\cdot q)_+\, dv = \dfrac{\pi R^4}{8} (p\cdot q )\, .  \]
\end{enumerate}
\end{lemma}

\begin{proof}
Item (\ref{item:i}) is immediate.
\\
Concerning item (\ref{item:ii}),
denote $J(p,q) = \int_{v\in V} (p\cdot v)(v\cdot q)_+\, dv$, then $J$ is symmetric:
\begin{eqnarray*}
 J(p,q) &=& \int_v (p\cdot v)_+(v\cdot q)_+\, dv - \int_v (p\cdot v)_-(v\cdot q)_+\, dv \\
& = & \int_v (p\cdot v)_+(v\cdot q)_+\, dv - \int_w (- p\cdot w)_-(-w\cdot q)_+\, dw \\
& = & \int_v (p\cdot v)_+(v\cdot q) \, dv\, .
\end{eqnarray*}
 Moreover, $J(p,q)$ is linear w.r.t. $p$, so it is bilinear w.r.t. $(p,q)$. It remains to compute the associated quadratic form:
\begin{eqnarray*}
 J(p,p) & = & \int_v (p\cdot v) (v\cdot p)_+\, dv \\
& = & \frac12 \int_v (p\cdot v) (v\cdot p)\, dv \\
& = & \frac 12 p^T \left\{ \int_{v} v\otimes v\, dv \right\} p\, .
\end{eqnarray*}
Thanks to isotropy we obtain, 
\begin{eqnarray} \int_{v} v\otimes v\, dv & = & 
\left\{\int_{r=0}^R \int_{\theta = 0}^{2\pi} r^2 \cos^2 \theta \,  r  dr d\theta \right\} \Id \label{eq:isotropy}\\
&= & \pi \dfrac{R^4}4 \Id \nonumber\, .
\end{eqnarray}
Consequently we deduce
\[ J(p,q) = \dfrac{\pi R^4}{8} (p\cdot q)\, . \]

\end{proof}

\begin{proof}[Proof of Theorem \ref{thm:BU}]
We plan to evaluate explicitly the time evolution of the second moment w.r.t. to space variable $x$. 
We introduce the notation,
\[ I(t) = \frac 12 \iint_{\RR^2\times V} |x|^2 f(t,x,v)\, dvdx \, ,\] 
We differentiate  twice in time:
\begin{eqnarray*}
 \frac d{dt} I(t) & = & \int_x \int_v (x\cdot v) f(t,x,v)\, dv dx \\ 
 &&\quad + \frac12 \int_x |x|^2 \left\{\int_v \int_{v'} T[S](t,x,v,v') f(t,x,v') dv'dv - \int_v \int_{v'} T[S](t,x,v',v) 
 f(t,x,v)\, dv'dv \right\}\, dx \\
& = & \int_x \int_v (x\cdot v) f(t,x,v)\, dv dx \, ,\\
  \frac {d^2}{dt^2} I(t) & = & \int_x\int_v |v|^2 f(t,x,v)\, dx dv  \\
&& + \int_x \int _v \int_{v'} (x\cdot v) T[S](t,x,v,v') f(t,x,v') \, dv' dv dx - \int_x\int_v  (x\cdot v) \lambda[S](t,x,v) 
f(t,x,v)\, dv dx\, .
\end{eqnarray*}
With the particular choice for $T[S]$ given by \eqref{eq:turning kernel} (it does not depend on the anterior velocity $v'$), we 
obtain:
\begin{eqnarray*}
  \frac {d^2}{dt^2} I(t) & = & \int_x\int_v |v|^2 f(t,x,v)\, dv dx  \\
&& + \chi_0 \int_x \int _v  (x\cdot v) (v\cdot \nabla S)_+ \rho(t,x) \,  dv dx - \chi_0 \int_x\int_v  (x\cdot v) 
\left\{\int_{v'} (v'\cdot \nabla S)_+ \, dv'\right\} f(t,x,v)\, dv dx\, .
\end{eqnarray*}
Therefore, applying Lemma \ref{lem:averaging}, we get 
\begin{equation} \label{eq:twice}
 \frac {d^2}{dt^2} I(t) = \int_x\int_v |v|^2 f(t,x,v)\, dx dv   + \chi_0 \dfrac{\pi R^4}{8} \int_x  x\cdot \nabla S(t,x)  
 \rho(t,x) \, dx - \chi_0 \dfrac{2R^3}3 \int_x\int_v  (x\cdot v) |\nabla S|(t,x)   f(t,x,v)\, dv dx\, .
\end{equation}
The following computation is well-known within the theory of the Keller-Segel system \cite{BDP06}:
\begin{eqnarray*}
 \int_x  x\cdot \nabla S(t,x)  \rho(t,x) \, dx & = & - \frac1{2\pi} \int_x \int_y x\cdot\dfrac{x-y}{|x-y|^2} \rho(t,y)\rho(t,x)\, 
 dy dx  \\
& = & - \frac1{4\pi} \int_x \int_y (x-y)\cdot\dfrac{x-y}{|x-y|^2} \rho(t,y)\rho(t,x)\, dy dx  \\
& = & -\frac{M^2}{4\pi}\, . 
\end{eqnarray*}
Therefore we obtain from \eqref{eq:twice},
\begin{eqnarray}
 \frac {d^2}{dt^2} I(t) &\leq& R^2 M    -  \dfrac{\chi_0 R^4}{32} M^2  - \chi_0 \dfrac{2R^3}3 \int_x\int_v  (x\cdot v) |\nabla 
 S|(t,x)   f(t,x,v)\, dv dx \nonumber\\
& \leq & R^2 M\left(1 - \dfrac{\chi_0 R^2 M}{32}\right) - \chi_0 \dfrac{2R^3}{3} \int_x x\cdot j(t,x) |\nabla S|(t,x)  \, dx\, , 
\label{eq:virial}
\end{eqnarray}
where the current $j(t,x) = \int_v vf(t,x,v)\, dv $ satisfies
\[  \partial_t \rho + \nabla\cdot j = 0\, . \]
Introduce the notation,
\begin{equation} \delta = R^2M\left(\dfrac{\chi_0 R^2 M}{32} -1\right)\, , \label{eq:delta def} \end{equation}
which is positive by assumption \eqref{eq:small mass BU} of Theorem \ref{thm:BU}.

In spherical coordinates, we can compute exactly the  contribution of the remaining term in \eqref{eq:virial}, using the following identities (see Appendix):
\begin{equation*}
\begin{array}{l@{\qquad}l}
\displaystyle|\nabla S|(t,x) = |S'(t,r)|\, , 
& \displaystyle r|S'(t,r)| =  \int_{\lambda =0}^r \lambda \rho(t,\lambda)\, d\lambda\, , \medskip \\
\displaystyle j(t,x) =  j^\parallel(t,r)\frac x{r} + j^\perp(t,r)\frac {x^\perp}{r}  \, ,
& \displaystyle r j^\parallel(t,r) = - \frac \partial{\partial t} \int_{\lambda =0}^r \lambda \rho(t,\lambda)\, d\lambda = \frac \partial{\partial t} \int_{\lambda =r}^\infty \lambda \rho(t,\lambda)\, d\lambda \, .
\end{array}
\end{equation*}
Therefore, under the hypothesis of spherical symmetry we get
\begin{eqnarray*}
-  \int_{\RR^2} x\cdot j(t,x) |\nabla S|(t,x)  \, dx  &=& - 2\pi \int_{r=0}^\infty  r   j^\parallel(t,r)  |S'(t,r)|\, r dr \\
& = & -2\pi  \int_{r=0}^\infty  \frac \partial {\partial t}\left(\int_{\lambda =r}^\infty \lambda \rho(t,\lambda)\, d\lambda \right)  \left(\frac M{2\pi} - \int_{\lambda =r}^\infty \lambda \rho(t,\lambda)\, d\lambda \right)\,  dr \\
& = & -  M  \frac d{dt} \int_{r=0}^\infty  \int_{\lambda =r}^\infty \lambda \rho(t,\lambda)d\lambda \,  dr +  \pi \frac d{dt} \int_{r=0}^\infty \left(\int_{\lambda =r}^\infty \lambda \rho(t,\lambda)\, d\lambda \right)^2\, dr\, .
\end{eqnarray*}
Integrating once in time the inequality \eqref{eq:virial} leads to
\[
\dfrac d{dt} I(t)  \leq \dfrac{d}{dt} I(t)\Big| _{  t=0} - \delta t  +  \chi_0 \dfrac{2R^3}{3}\left(K(0)-  K(t)\right)\, , 
\]
where $K(t)$ is defined by
\begin{eqnarray*} K(t) &=&    M   \int_{r=0}^\infty  \int_{\lambda =r}^\infty \lambda \rho(t,\lambda) \,  d\lambda dr -  \pi \int_{r=0}^\infty \left(\int_{\lambda =r}^\infty \lambda \rho(t,\lambda)\, d\lambda \right)^2\, dr \\
& = &   \frac M2   \int_{r=0}^\infty  \int_{\lambda =r}^\infty \lambda \rho(t,\lambda)\,d\lambda  dr + \pi \int_{r=0}^\infty \left(\frac M{2\pi} - \int_{\lambda =r}^\infty \lambda \rho(t,\lambda)\, d\lambda  \right)\left(\int_{\lambda =r}^\infty \lambda \rho(t,\lambda)\, d\lambda \right) \, dr \, .
\end{eqnarray*}
Thus $K(t)$ is obviously a nonnegative quantity. It is worth noticing that $K(t)$ is finite provided that the density $\rho$ has a finite second moment. Indeed we have 
\begin{eqnarray*}  K(t)  
& \leq & M \int_{r=0}^\infty \int_{\lambda =r}^\infty \lambda \rho(t,\lambda)\, d\lambda  dr \\
& = & M \int_{\lambda = 0}^\infty \lambda^2 \rho(t,\lambda)\, d\lambda\\
& \leq & \dfrac M{2\pi}  \int_{\RR^2} |x|\rho(t,x)\, dx\, .
\end{eqnarray*}
We deduce in particular using the Cauchy-Schwarz inequality,
\begin{equation} \label{eq:K(t)}
K(t)\leq \dfrac1{2\pi}M^{3/2} \sqrt{2I(t)}\, .
\end{equation}

\begin{remark}
 It seems surprising that the contribution of the loss term to the virial identity can be written as the derivative of a 
 nonpositive quantity. It is not surprising however if we notice that in spherical coordinates, this contribution is the scalar 
 product between $r j^\parallel(t,r)$ and $r S'(t,r)$ satisfying respectively
\begin{equation*}
\left\{\begin{array}{l}
\displaystyle \dfrac \partial{\partial r}\big( r j^\parallel(t,r)\big) = - \dfrac \partial{\partial t}\big( r\rho(t,r)\big)\, , \medskip \\
\displaystyle \dfrac \partial{\partial r} \big(r S'(t,r)\big) =   r\rho(t,r) \, .
\end{array}\right.
\end{equation*}
On the other hand this property still holds true when the turning kernel contains a linear part (see \eqref{eq:derivative of 
I(t)}).
\end{remark}

We end up eventually with
\begin{equation} 
\label{eq:integrate virial}
\dfrac d{dt} I(t) \leq \iint_{\RR^2\times V}  (x\cdot v) f_0(x,v)\, dv dx- \delta t + \chi_0 \dfrac{2R^3}3 K(0) \, .  
\end{equation}
This proves that the second moment formally vanishes in finite time. therefore a singularity necessarily forms before this time, otherwise it would contradict local existence stated in the Appendix.
\end{proof}

\subsubsection*{The case of $V = \Sph(0,R)$}
In the case where the set of admissible velocities is the sphere of radius $R$, we can adapt the proof above to demonstrate 
that blow-up occurs if 
\begin{equation} \tilde\delta = R^2M\left(\dfrac{\chi_0 R M}{8} -1\right)\, , \label{eq:delta sphere} \end{equation}
is a positive quantity.
The important modification arises in the constants evaluated in Lemma \ref{lem:averaging}. We can adapt the computations
in that Lemma  
to obtain that for any $(p,q)\in \RR^2\times \RR^2$, we have
\[ \int_{V} (p\cdot v)(v\cdot q)_+\, dv = \dfrac{\pi R^3}{2} (p\cdot q )\, .  \]
As a consequence, we get that the key differential virial inequality \eqref{eq:integrate virial} becomes in this new setting,
\begin{equation} 
\label{eq:integrate virial bis}
\dfrac d{dt} I(t) \leq \iint_{\RR^2\times V}  (x\cdot v) f_0(x,v)\, dv dx- \tilde \delta t + \chi_0 {2R^2}K(0) \, ,  
\end{equation}
This concludes the adaptation to the case of the sphere.

\subsection{Blow-up including chemical degradation ($\alpha>0$)}

It is now classical that for results concerning blow-up in the parabolic Keller-Segel system, the additional contribution of chemical degradation
\[ -\Delta S + \alpha S = \rho\, , \]
does not change dramatically the flavor of the results. It affects the blow-up criterion, however, and so it does in our situation.

In order to study the influence of the chemical degradation, we shall estimate carefully the size of the corrective terms that 
come from the difference between the Poisson kernel and the Bessel kernel, respectively
\begin{equation} 
 \label{eq:kernels}
B_0(z) = \dfrac{1}{2\pi}  \log \dfrac 1{ |z|}\, , \quad B_\alpha(z) = \dfrac1{4\pi}\int_{t=0}^{+\infty}\dfrac 1t\  e^{-\frac{|z|^2}{4t} - \alpha t}\ dt\, .
\end{equation}

\begin{lemma}
\label{lem:error kernel}
There exists a universal constant $\mathcal C$ such that 
\begin{equation}
\|\nabla B_\alpha  - \nabla B_0 \|_\infty\leq \sqrt{\alpha} \mathcal C\, . 
\end{equation}
\end{lemma}

\begin{remark}
Observe that the scaling $\sqrt{\alpha}$ comes naturally from knowing the estimate between $B_1$ and $B_0$.
\end{remark}

\begin{proof}
We give an argument in Fourier space. In fact $\hat{B}_\alpha(\xi) = (\alpha + |\xi|^2)^{-1}$, and we get:
\begin{eqnarray*} 
\| \nabla B_\alpha  - \nabla B_0 \|_\infty &\leq& \dfrac1{2\pi} \left\| \widehat{\nabla B}_\alpha  -  \widehat{\nabla B}_0 \right\|_1 \\
&\leq & \dfrac1{2\pi} \int_{\RR^2} |\xi| \left( \dfrac1{ |\xi|^2} - \dfrac1{\alpha +|\xi|^2} \right) \, d\xi \\
&\leq & \dfrac1{2\pi} \int_{\RR^2}   \dfrac{\alpha }{|\xi|\left(\alpha + |\xi|^2 \right)} \, d\xi \\
&\leq & \dfrac{\sqrt{\alpha}}{2\pi} \int_{\RR^2}   \dfrac{1 }{|\zeta|\left(1 + |\zeta|^2 \right)} \, d\zeta\, .
\end{eqnarray*}

\end{proof}

\begin{proof}[Proof of Corollary \ref{cor:BU}] Following 
\eqref{eq:twice} and subsequent lines we are able to obtain a perturbated virial identity for 
the second space moment of the density:
\begin{equation}
\begin{split}
 \frac {d^2}{dt^2} I(t) = \int_x\int_v |v|^2 f(t,x,v)\, dx dv   + \chi_0 \dfrac{\pi R^4}{8} \int_x  x\cdot \nabla  \tilde S(t,x)  
 \rho(t,x) \, dx - \chi_0 \dfrac{2R^3}3 \int_x\int_v  (x\cdot v) |\nabla  \tilde S|(t,x)   f(t,x,v)\, dv dx \\
 +  \chi_0 \dfrac{\pi R^4}{8} M \int_x  x\cdot \left(  \nabla    S(t,x) -  \nabla  \tilde S(t,x) \right) \rho(t,x) \, dx - \chi_0 
 \dfrac{2R^3}3 \int_x\int_v  (x\cdot v) \left(  |\nabla   S|(t,x) -  |\nabla  \tilde S|(t,x) \right)  f(t,x,v)\, dv dx \, ,
\end{split}
\label{eq:perturbated twice}
\end{equation}
where $\nabla \tilde S$ is defined by $-\Delta \tilde S = \rho$ as above (cf. the case $\alpha = 0$).
The first line of \eqref{eq:perturbated twice} can be explicitly computed as before. It remains to estimate the error terms (on the second line) using Lemma \ref{lem:error kernel}. In fact we have using Young's inequality:
\begin{eqnarray*}
 \left| \chi_0 \dfrac{\pi R^4}{8} \int_x  x\cdot \left(  \nabla    S(t,x) -  \nabla  \tilde S(t,x) \right) \rho(t,x) \, dx \right| & \leq & \chi_0 M \dfrac{\pi R^4}8  \int_x |x| \left\| \nabla B_\alpha - \nabla B_0 \right\|_\infty \rho(t,x)\, dx \\
&\leq&  \sqrt{\alpha}\chi_0 M \dfrac{\pi R^4}8 \mathcal{C} \int_x |x| \rho(x)\, dx
\\
&\leq & \sqrt{\alpha} \chi_0 M \dfrac{\pi R^4}8 \mathcal{C} \sqrt{M} \sqrt{2 I(t)}\, .
\end{eqnarray*}
In the same way we get,
\begin{eqnarray*}
 \left| \chi_0 \dfrac{2R^3}3 \int_x\int_v  (x\cdot v) \left(  |\nabla   S|(t,x) -  |\nabla  \tilde S|(t,x) \right)  f(t,x,v)\, dv dx \right| 
& \leq & \chi_0 \dfrac{2R^4}3   \int_x |x| \left| \nabla   S(t,x) -  \nabla \tilde S(t,x) \right| \rho(t,x)\, dx \\
&\leq&  \sqrt{\alpha}\chi_0 M \dfrac{2R^4}3 \mathcal{C} \int_x |x| \rho(x)\, dx
\\
&\leq & \sqrt{\alpha} \chi_0 M \dfrac{2R^4}3 \mathcal{C}  \sqrt{M} \sqrt{2 I(t)}\, .
\end{eqnarray*}
Therefore we end up with the following integro-differential inequality instead of \eqref{eq:integrate virial}:
\begin{equation}
\label{eq:corrected virial}
\dfrac d{dt} I(t) \leq \int_x\int_v (x\cdot v) f_0(x,v)\, dv dx  - \delta t+\chi_0 \dfrac{2R^3}3 K(0) + 
\sqrt{\alpha} \chi_0 M^{3/2}R^4 \mathcal{C} \int_{\tau=0}^t \sqrt{I(\tau)}\, d\tau \, .
\end{equation}
Recall that $\delta>0$ is defined in \eqref{eq:delta def}. The universal constant $\mathcal{C}$ is now fixed for the rest of this 
proof. Denote 
\begin{align*}
& \mu_0 =  \int_x \int_v (x\cdot v) f_0(x,v)\, dv dx + \chi_0 \dfrac{2R^3}3 K(0) \, ,  \\
& \eta = \sqrt{ \alpha}\chi_0 M^{3/2}R^4 \mathcal{C} \, . 
\end{align*}
Integrating once more and inverting the order of integration we get
\begin{align*}
 I(t)&\leq I(0) + \mu_0 t - \frac{\delta}{2}t^2  + \eta \int_{s=0}^{t}\int_{\tau = 0}^{s}\sqrt{I(\tau)}\, d \tau ds\\
&\leq I(0) + \mu_0 t - \frac{\delta}{2}t^2  + \eta \int_{\tau=0}^{t}(t-\tau)\sqrt{I(\tau)}\, d \tau \\
&\leq I(0) + \mu_0 t - \frac{\delta}{2}t^2  + \eta  \int_{\tau=0}^{t}\left(\frac{\eps}{2}(t-\tau)^2 + \frac{1}{2\eps}I(\tau) 
\right)\, d \tau \\
&\leq I(0) + \mu_0 t - \frac{\delta}{2} t^2 + \frac{\eta\eps}{6} t^3 + \frac{\eta}{2\eps}
\int_{\tau=0}^{t} I(\tau)\, d \tau\, .\label{new4}
\end{align*}
We obtain as a consequence,
\begin{equation*}
\int_{\tau=0}^t I(\tau)\, d\tau \leq e^{\eta t/2\eps} \int_{\tau=0}^t e^{-\eta \tau/2\eps}\left( I(0) + \mu_0 \tau - 
\frac{\delta}{2} \tau^2 + \frac{\eta\eps}{6} \tau^3  \right)\, d\tau\, . 
\end{equation*}
A sufficient condition for a contradiction to occur is that the right-hand side becomes negative as $t\to\infty$. For this 
purpose, compute
\begin{eqnarray*}
\dfrac{ \eta}{2\eps} \int_{\tau=0}^\infty e^{-\eta \tau/2\eps}\left( I(0) + \mu_0 \tau - \frac{\delta}{2} \tau^2 + 
\frac{\eta\eps}{6} \tau^3  \right)\, d\tau & = &  \int_{\sigma=0}^\infty e^{-\sigma}  \left( I(0) + \dfrac{2\eps\mu_0}{\eta} 
\sigma - \frac{2\eps^2\delta}{\eta^2} \sigma^2 + \frac{4 \eps^4}{3\eta^2} \sigma^3  \right)\, d\sigma \\
 & = & I(0) + \dfrac{2\eps\mu_0}{\eta} - \frac{4\eps^2\delta}{\eta^2} + \frac{8 \eps^4}{\eta^2} \, .
\end{eqnarray*}
Choose for instance $\eps = \sqrt{\delta}/2$. Then the quantity $\int_{\tau=0}^t I(\tau) d\tau$ eventually vanishes if the following condition is fulfilled:
\begin{equation}
\label{eq:suff. cond.}
I(0) + \dfrac{ \sqrt{\delta}\mu_0}{\eta} < \dfrac{\delta^2}{2\eta^2}\, . 
\end{equation}

To conclude the proof of Corollary \ref{cor:BU}, observe that, resulting from \eqref{eq:K(t)} we get:
\begin{equation*}
\mu_0 \leq R \sqrt{M} \sqrt{2I(0)} + \dfrac{\chi_0 M^{3/2}R^3}{3\pi} \sqrt{2I(0)}\, ,
\end{equation*}
so that the necessary condition  \eqref{eq:suff. cond.} might be replaced by the stronger criterion:
\begin{equation*}
\eta^2 \dfrac{2I(0)}{\delta} + 2\left( R\sqrt{M} + \dfrac{\chi_0 M^{3/2}R^3}{3\pi}\right)\sqrt{\eta^2 \dfrac{2I(0)}{\delta}} < \delta \, . 
\end{equation*} This criterion can be read as $X + 2A\sqrt{X}<\delta$, which is equivalent to $\sqrt{X}<\sqrt{\delta + A^2} - A$.
Therefore, the following criterion is a necessary condition for solutions to blow-up after finite time,
\begin{align*}
&\eta^2 \dfrac{2I(0)}{\delta}<  A^2\left(  \dfrac{\delta}{A^2} + 2 - 2\sqrt{1+ \dfrac{\delta}{A^2} } \right)\, ,\\
& A = R\sqrt{M} + \dfrac{\chi_0 M^{3/2}R^3}{3\pi}\, .
\end{align*}
\end{proof}

\subsection{Drift-diffusion limit and blow-up (case $\alpha=0$)}\label{ddl}


In the Introduction, we derive formally the parabolic Keller-Segel system from the kinetic system with suitable turning kernel \eqref{eq:turning kernel bis}.

Repeating the  virial computation in this case leads to 
\begin{eqnarray*}
\epsilon^2 \frac {d^2}{dt^2} I_\epsilon(t) & = &\int_x\int_v |v|^2 f_\epsilon(t,x,v)\, dvdx + \dfrac1\epsilon \int_x\int_v (x\cdot v) \rho_\epsilon(t,x) F(v)\, dv dx - \dfrac1\epsilon \int_x  x\cdot j_\epsilon(t,x)\, dx  -  \dfrac{\chi_0 R^4}{32} M^2 - \chi_0 \dfrac{2R^3}3\dfrac{d}{dt} K_\epsilon(t) \\
&= & \int_x\int_v |v|^2 f_\epsilon(t,x,v)\, dvdx -  \dfrac{\chi_0 R^4}{32} M^2 - \dfrac1{2\epsilon} \int_x\left( \nabla |x|^2\right) j_\epsilon(t,x)\, dx - \chi_0 \dfrac{2R^3}3 \dfrac{d}{dt} K_\epsilon(t) \, ,
\end{eqnarray*}
because $\int_V vF(v)\, dv = 0$ by assumption. To conclude as above that the remaining term is the derivative of a nonpositive quantity, observe that 
\begin{eqnarray}  
-   \int_{\RR^2}\left( \nabla |x|^2\right) j_\epsilon(t,x)\, dx &=& \int_{\RR^2} |x|^2 \nabla\cdot j_\epsilon(t,x)\, dx \nonumber \\
& = & - \epsilon \int_{\RR^2} |x|^2 \partial_t \rho_\epsilon(t,x)\, dx \nonumber\\
& = & - \epsilon \dfrac d{dt} \int_{\RR^2} |x|^2 \rho_\epsilon(t,x)\, dx\, . \label{eq:derivative of I(t)}
\end{eqnarray}
Integrating once in time we obtain
\begin{equation} \epsilon^2 \dfrac d{dt} I_\epsilon(t) \leq \epsilon \int_x\int_v (x\cdot v)f_\epsilon(x,v)\, dv dx + R^2 M\left(1 - \dfrac{\chi_0 R^2 M}{32}\right) t +  I_\epsilon(0)-I_\epsilon(t) + \chi_0 \dfrac{2R^3}3\left(K_\epsilon(0) - K_\epsilon(t)\right)\, .   \label{eq:parabolic BU}\end{equation}
Arguing as before, we conclude that the solution blows-up in finite time under the same assumption as Theorem \ref{thm:BU}. 

On the other hand, we know precisely the blow-up criterion for the parabolic limit. It depends upon the choice of the relaxation function $F(v)$. This function can vary between two extremal choices: a Dirac mass at zero, and a Dirac mass spread on the sphere $\{|v|= R\}$.   Consider for example $F(v)\equiv \dfrac1{|V|}\mathbf{1}_V$. The parabolic limit \eqref{eq:parabolic limit} writes 
\[ \partial_t\rho  = \nabla\cdot \left(  \dfrac{R^2}{4} \nabla \rho \right) -  \dfrac{\chi_0\pi R^4 }{8} \nabla\cdot\left(  \rho \nabla S    \right) \, ,
\]
thanks to isotropy \eqref{eq:isotropy}. Thus the blow-up criterion for finite-time blow-up reads 
\[ M>\dfrac{ 16  }{\chi_0 R^2}\, , \]
which differs from the kinetic criterion \eqref{eq:small mass BU} by a factor $2$ (the two criteria actually match in the case 
where $F(v)$ is the normalized Dirac mass on the sphere of radius $R$). To fill this gap it is necessary to reconsider the 
parabolic limit a $\epsilon \to 0$. In fact the positive contribution $R^2M$ in \eqref{eq:parabolic BU} comes from the 
upper-bound of $\iint_{\RR^2\times V} |v|^2 f_\epsilon(tx,v)dvdx$. One may notice that for small $\epsilon$, $f_\epsilon(t,x,v)$ 
gets close to $\rho(t,x)F(v)$, and this upper-approximation is not sharp (except in the case where $F(v)$ is especially the 
Dirac mass on the sphere of radius $R$). Therefore we may replace $\iint_{\RR^2\times V} |v|^2 f_\epsilon(t,x,v)dvdx$ by $ 
M\int_V |v|^2F(v)\, dv  $. To finish with, observe that the diffusion tensor at the parabolic limit can be calculated due to 
rotational invariance, 
\[  \int_V v\otimes vF(v)\, dv = \dfrac12 \left(\int_V |v|^2 F(v)\, dv\right) \Id\, . \]
Under these considerations, the kinetic and the parabolic criterions do coincide.


\section{Global existence for small mass}
\label{sec:existence}

The aim of this section is to prove global existence for the kinetic model \eqref{kinmodel} in the spherically symmetric case 
under the small mass condition stated in Theorem \ref{thm:existence}.

Spherical symmetry is used in a crucial way in the following estimate:
\begin{equation}\label{eq:gradient estimate} (v\cdot \nabla S)_+ = |S'(r)| \left(v\cdot \dfrac{x}{|x|}\right)_-\leq \left(v\cdot 
\dfrac{x}{|x|}\right)_- \dfrac M {2\pi|x|}\, . \end{equation} To justify \eqref{eq:gradient estimate}, just write
\[ -S'(r) = \dfrac 1r\int_{\lambda =0}^r \lambda\rho(\lambda) \, d\lambda\leq \dfrac M {2\pi r}\, . \]


Recall that $0<\gamma<1$ is an exponent given by some upper-bound on the initial data. Introduce the auxiliary function 
\begin{equation} 
\label{eq:k}
k(x,v) = k_0 \left| x - \left(\dfrac v{|v|}\cdot x\right)_+ \dfrac v{|v|}  \right|^{-\gamma}= k_0 \left\{\begin{array}{ll} 
 |x|^{-\gamma} \quad& \mbox{if} \quad (v\cdot x)<0  \smallskip \\
 \left| \Pi_{v^\perp}(x)\right|^{-\gamma} \quad& \mbox{if} \quad (v\cdot x)>0  \\
\end{array} \right.
\, ,
\end{equation}
where $\Pi_{v^\perp} = \Id - {v\otimes v}/{|v|^2}$ denotes the orthogonal projection onto $v^\perp$.
We shall prove in fact that, as soon as  $f_0(x,v)\leq k(x,v)$, then it holds true that $f(t,x,v)\leq k(x,v)$ for all time $t>0$. 
This will be achieved through a comparison principle adapted to our context. 

\begin{proposition}[Properties of the auxiliary function $k$]
\label{rem:k}  
\begin{enumerate}[(i)] 
 \item The function $k(x,v)$
belongs to $L^{p}_{\mathrm{loc},x}L^q_v$ 
provided that $p<2/\gamma$ and $q<1/\gamma$.
Moreover it is a $\mathcal C^1$ function of $(x,v)$ in $\left(\RR^2\times V\right)\setminus\left\{(x,v)| x \upuparrows v 
\right\}$.
\item
For all $(x,v)$ we have $k(x,v)\geq k_0 |x|^{-\gamma}$. 
Therefore the initial comparison $f_0(x,v)\leq k(x,v)$ is guaranteed by the assumptions of Theorem \ref{thm:existence}.
\end{enumerate}
\end{proposition}
\begin{proof} Working exactly as in the proof of \eqref{rho_of_k} below we find that
$\int_V k(x,v)^q dv = k_{0}^{q}\frac{\abs{V}}{2}\Omega(\gamma q) \abs{x}^{-\gamma q}$, where $\Omega(\gamma q)$ 
(see \eqref{eq:Omega})
is finite  thanks to $q<1/\gamma$. Taking the $L^p_x$-norm gives the first assertion.  
 The inequality $k(x,v)\geq k_0 |x|^{-\gamma}$ is obvious when $x\cdot v<0$, and follows from 
 $\abs{\Pi_{v^\perp}(x)}=\abs{x}|\sin\theta|$, where $\theta=\angle(x,v)$, when $x\cdot v<0$.
\end{proof}

The crucial Lemma, which motivates the definition \eqref{eq:k} of $k(x,v)$  is the following one.

\begin{lemma}
\label{lem:sursolution}
Assume \eqref{eq:small mass}.
The function $k(x,v)$ is a supersolution of \eqref{kinmodel}, in the sense that: 
\begin{eqnarray} v\cdot \nabla_x k &=& \left(v\cdot \dfrac{x}{|x|}\right)_- \dfrac \gamma{|x|} k(x,v)\nonumber \\
& \geq & \chi_0 \left(v\cdot \dfrac{x}{|x|}\right)_- \dfrac1{|x|} \dfrac{M}{2\pi} \int_{v'} k(x,v')\, dv'\, . 
\label{eq:sursolution}
\end{eqnarray}
\end{lemma}

\begin{proof}
First we evaluate 
\begin{align}
\int_{v'} k(x,v') \, dv' &= k_0|x|^{-\gamma} \int_{\{v'\in V|(x\cdot v')<0\}} \, dv + 
k_0 |x|^{-\gamma} \int_{\{v'\in V|(x\cdot v')>0\}} \left|\dfrac{\Pi_{v'^\perp}(x)}{|x|} \right|^{-\gamma} \, dv'
\nonumber\\
& =  k_0|x|^{-\gamma} \dfrac{|V|}2 + k_0 |x|^{-\gamma} \int_{\nu=0}^R \int_{\theta = -\pi/2}^{\pi/2}  
\left| \sin \theta  \right|^{-\gamma} \, \nu d\nu d\theta \nonumber\\
& =  k_0|x|^{-\gamma} \dfrac{|V|}2 \Omega(\gamma)   \nonumber\\
& \leq  k(x,v) \dfrac{|V|}2 \Omega(\gamma)  \, . \label{rho_of_k}
\end{align}
where we have used part (ii) of Proposition \ref{rem:k}. 

In order to prove \eqref{eq:sursolution}, let us distinguish between $(v\cdot x)>0$ and $(v\cdot x)<0$. In the former case we have
\begin{eqnarray*}
  v\cdot \nabla_x k &=& -\gamma k_0 v\cdot \left(\Pi_{v^\perp}\circ \Pi_{v^\perp}\right)(x) |\Pi_{v^\perp}(x)|^{-\gamma-2} \\
& = & 0\, .
\end{eqnarray*}
because $\Pi_{v^\perp}$ is a linear symmetric operator whose image is orthogonal to $v$.
%
In the latter case $(v\cdot x)<0$ we  have
\begin{eqnarray*} v\cdot \nabla_x k &=& -\gamma  k_0(v\cdot x) |x|^{-\gamma-2} \\
& = & \left(v\cdot \dfrac{x}{|x|}\right)_- \dfrac{\gamma}{|x|} k(x,v) \\
& \geq & \left(v\cdot \dfrac{x}{|x|}\right)_-\dfrac{2\gamma}{|V|\Omega(\gamma)}  \dfrac{1}{|x|}\int_{v'} k(x,v') \, dv'
\, . 
\end{eqnarray*}
and we conclude the proof by using the smallness condition \eqref{eq:small mass} in the assumptions 
of Theorem \ref{thm:existence}. 
\end{proof}

\begin{definition}[Set of admissible exponents]
\label{def:admissible}
A couple of exponents $(p,q)$ is said to be admissible if it satisfies
\[2< p < \frac2\gamma\,, \quad 1< q < \frac1\gamma\, , \quad 0\leq\dfrac1q - \dfrac1p < \dfrac12\,,\quad \dfrac{q'}{p'}> \dfrac{1-\gamma/2}{1-\gamma}\, .\]
This set is nonempty as it can be seen when $(p,q)\to (2^+,1^+)$.
\end{definition}

\begin{lemma}[$L^p_xL^q_v$ regularity is ensured by comparison]
\label{lem:LpLq comparison} Let $(p,q)$ be a set of admissible exponents.
Assume that $f(T,x,v)$ lies below $k(x,v)$. Then $f(T,x-t v,v)$ belongs to $L^p_xL^q_v$ for all $t>0$. 
\end{lemma}

\begin{proof}
We shall first prove that 
\begin{equation}\label{comparison21}
\forall \, x\neq 0\quad \norm{k(x-tv,v)}{L^q_v}\leq C(q,\gamma,k_0,V) \abs{x}^{-\gamma }\, .
\end{equation}
For this purpose, let us decompose
\begin{equation}\label{comparison22}
\int_V k(x-tv,v)^q\, dvk_0^q \int_{\{v\in V|(x-t  v)\cdot v< 0\}} |x - t  v|^{-\gamma q}\, dv 
+ k_0^q \int_{\{v\in V|(x- t  v)\cdot v> 0\}} |\Pi_{v^\perp}(x-t  v)|^{-\gamma q}\, dv\, .
\end{equation}
For the second contribution in the right-hand-side above we use 
$\Pi_{v^\perp}(x-t v)=\Pi_{v^\perp}(x)$ and
$\abs{\Pi_{v^\perp}(x)}=\abs{x}|\sin\theta|$,
where $\theta=\angle(x,v)$, to get
\begin{eqnarray}
 k_0^q \int_{\{v\in V|(x- t  v)\cdot v> 0\}} |\Pi_{v^\perp}(x-t  v)|^{-\gamma q}\,
dv&\leq &
k_0^q \int_{V} |\Pi_{v^\perp}(x)|^{-\gamma q}\,
dv\nonumber\\
&\leq& k_0^q \abs{x}^{-\gamma q} \dfrac{R^2}2 \int_{\theta = -\pi}^\pi \abs{\sin\theta}^{-\gamma q}\, d\theta\nonumber\\
&=& C(q,\gamma,k_0,V)\abs{x}^{-\gamma q}\label{comparison23}\, .
\end{eqnarray}
To estimate the first contribution in the right-hand-side of \eqref{comparison22} we distinguish between two
cases. \\
If $\abs{x}\leq 2 tR$ then 
$V\subset \{v\in \RR^2| \abs{ {x}/{t}-v}\leq 3R\}$, therefore
\begin{eqnarray*}
  k_0^q \int_{\{v\in V|(x-t  v)\cdot v< 0\}} |x - t  v|^{-\gamma q}\, dv&\leq &
 k_0^q t^{-\gamma q}\int_{V} \abs{\frac{x}{t} -   v}^{-\gamma q}\, dv\\
&\leq&  k_0^q t^{-\gamma q}\int_{\{w\in \R^2| \abs{w}\leq 3R\}} \abs{w}^{-\gamma q}\,
dv\\
&\leq & k_0^q  \left(\frac{\abs{x}}{2R} \right)^{-\gamma q}C(q,\gamma, V)\\
&\leq &C(q,\gamma, k_0,V) \abs{x}^{-\gamma q}\, .
\end{eqnarray*}
\noindent
On the other hand, if $\abs{x}\geq 2 t R$ then we use
$\abs{x-tv}\geq \abs{x}-t\abs{v}\geq {\abs{x}}/{2}$ to obtain eventually,
\begin{equation}\label{comparison24}
k_0^q \int_{\{v\in V|(x-t  v)\cdot v< 0\}} |x - t  v|^{-\gamma q}\, dv \leq 
 k_0^q \, 2^{\gamma q} \abs{x}^{-\gamma q}\abs{V}\, .
\end{equation}

In a second step we split the $L_x^pL_v^q$ norm of $f(T,x-tv,v)$ into a short-range part (in space) and a long-range part, as follows
 \begin{equation*}\label{comparison25}
 \norm{f(T,x-tv,v)}{L^p_x L^q_v}\leq 
\norm{f(T,x-tv,v)\mathbbm{1}_{\abs{x}\leq 1}}{L^p_x L^q_v} 
+ \norm{f(T,x-tv,v)\mathbbm{1}_{\abs{x}\geq 1}}{L^p_x L^q_v}\, .
\end{equation*}
For the short-range contribution $\abs{x}\leq 1$ we use \eqref{comparison21} -- which is a combination of \eqref{comparison23} and \eqref{comparison24} -- to get
\begin{equation*}
\norm{f(T,x-tv,v)\mathbbm{1}_{\abs{x}\leq 1}}{L^p_x L^q_v} \leq 
\norm{k(x-tv,v)\mathbbm{1}_{\abs{x}\leq 1}}{L^p_x L^q_v} \leq 
C(q,\gamma, k_0,V)
 \norm{\abs{x}^{-\gamma}\mathbbm{1}_{\abs{x}\leq 1}}{L^p_x }\leq C(p,q,\gamma, k_0,V)\, .
\end{equation*}
because $\gamma p < 2$.\\
For the long-range contribution $\abs{x}\geq 1$ we introduce a pair of auxiliary exponents $(P,Q)$ such that,
\begin{equation*}\label{PQ}
\frac{2}{\gamma}<P<\infty\, ,\quad 1< 
Q<\frac{1}{\gamma}\, , \quad \frac{P'}{Q'}=\frac{p'}{q'}\, .
\end{equation*}
We shall ensure that such a choice of $(P,Q)$ exists: in fact when $Q\to (1/\gamma)^-$ we have 
\[  \dfrac{1-\gamma/2}{1-\gamma} <\dfrac{q'}{p'} = \dfrac{Q'}{P'} \to \dfrac{1}{P'}\dfrac1{1-\gamma}\, . \]
Therefore we can find $P'<(1-\gamma/2)^{-1}$, {\em i.e.} $P>2/\gamma$.\\
Define $\theta={P'}/{p'}={Q'}/{q'}$. We have the interpolation relation $L^p_xL^q_v = [L^1_{x,v},L^{P}_xL^{Q}_v]_{1-\theta,\theta}$, because
\begin{equation*}\label{PQ1}
 \frac{1}{p}=1-\theta + \frac{\theta}{P}, \ \  \frac{1}{q}=1-\theta +
\frac{\theta}{Q}\, .
\end{equation*}
As a consequence,
\begin{eqnarray*}
 \norm{f(T,x-tv,v)\mathbbm{1}_{\abs{x}\geq 1}}{L^p_x L^q_v}&
\leq& \norm{f(T,x-tv,v)\mathbbm{1}_{\abs{x}\geq 1}}{L^1_x L^1_v}^{1-\theta}
\norm{f(T,x-tv,v)\mathbbm{1}_{\abs{x}\geq 1}}{L^P_x L^Q_v}^{\theta}\\
&\leq& M^{1-\theta} \norm{k(x-tv,v)\mathbbm{1}_{\abs{x}\geq 1}}{L^P_x L^Q_v}^{\theta}\\
&\leq& C'(p,q,\gamma,k_0,V,M) \norm{\abs{x}^{-\gamma}\mathbbm{1}_{\abs{x}\geq 1}}{L^P_x
}^{\theta}
\\
&\leq& C(p,q,\gamma,k_0,V,M)\, .
\end{eqnarray*}

\end{proof}

\begin{proof}[Proof of Theorem \ref{thm:existence}]

We obtain from Lemma \ref{lem:sursolution} and the elliptic bound \eqref{eq:gradient estimate} (which holds true in the spherically symmetric framework) the following crucial estimate:
\begin{equation}
v\cdot \nabla_x k \geq \chi_0 \left(v\cdot \dfrac{x}{|x|}\right)_- \dfrac M{2\pi|x|}  \int_{v'} k(x,v')\, dv' \geq \chi_0 (v\cdot \nabla S)_+ \int_{v'} k(x,v')\, dv'\, .
\end{equation}
Therefore we get the following differential inequality, well-suited for proving a sort of comparison principle:
\begin{equation}
\label{eq:comparison inequality}
\partial_{t} (f-k) + v\cdot\nabla_{x} (f-k) \leq  \chi_0(v\cdot \nabla S)_+ \left(\rho(x) - \int_{v'} k(x,v')\, dv'\right)\, . 
\end{equation} 
However the non-local nature of the right-hand side requires more regularity, and a local in time estimate 
(obtained in the Appendix) will enter into the game.
Due to the lack of integrability at infinity of the reference function $k$, we aim to localize in space
 such a partial differential inequality. 
In order to do so, multiply \eqref{eq:comparison inequality} by the test function $\varphi(x) = \exp\left(-\sqrt{1+|x|^2}\right)$ 
which satisfies: 
\begin{equation}
 \label{eq:phi}
\forall\, x\quad \left|\nabla \varphi(x) \right| = \left| \dfrac x{\sqrt{1+|x|^2}} \right| \varphi(x) \leq \varphi(x)\, .  
\end{equation}
Introduce the notation:
\[K(x) = \int_{v'\in V} k(x,v')\, dv'\, .\]
We get the following localized partial differential inequality,
\begin{equation}
\label{eq:comparison inequality bis} 
\partial_{t} \left((f-k)\varphi\right) + v\cdot\nabla_{x} \left((f-k)\varphi\right) \leq  
\chi_0(v\cdot \nabla S)_+ \left(\rho  - K \right) \varphi  +  (f-k)v\cdot  \nabla \varphi \, . 
\end{equation}

Introduce $\bP_\eps$ a sub-approximation of the positive part $(\cdot)_+$. Namely we choose $\bP_\eps(z) = \eps \bP_1(z/\eps)$, 
where 
\[ \bP_1(z) = \left\{\begin{array}{l@{\quad\mbox{if}\quad}l} 0 & z\leq 0 \\
z^2/2 & 0\leq z \leq 1 \\
z - 1/2 & z\geq 1 \end{array}\right. \, .\]
In particular, $\bP_\eps(z)$ is identically zero on $\{z\leq 0\}$, $0\leq \bP'_\eps(z)\leq 1$ and moreover
\begin{equation} \label{eq:P approximation} \forall\, z\quad |z| \bP'_\eps(z)\leq 2 \bP_\eps(z)  \, .\end{equation} 
Multiplying \eqref{eq:comparison inequality bis} by $\bP'_\eps((f-k)\varphi)$ we obtain
\begin{eqnarray*}
\partial_{t} \bP_\eps((f-k)\varphi) + v\cdot\nabla_{x} \bP_\eps((f-k)\varphi) &\leq&  
\chi_0(v\cdot \nabla S)_+ \left(\rho  - K \right)\varphi \bP_\eps'((f- k)\varphi) +   (f-k) \left(v\cdot\nabla \varphi\right) 
\bP_\eps'((f- k)\varphi) \\
&\leq&  \chi_0 (v\cdot \nabla S)_+ \left(\rho  - K \right)_+\varphi + 2 R \bP_\eps((f-k)\varphi)
\, ,
\end{eqnarray*} 
due to \eqref{eq:phi} and \eqref{eq:P approximation}.
Therefore we obtain, thanks to Duhamel's representation of the solution after a given time $T$,
\begin{equation*}
\begin{split}
&\bP_\eps((f-k)\varphi)(T+t ,x,v) \leq e^{2Rt } \bP_\eps((f-k)\varphi)(T,x-t  v,v)\\ 
&\qquad+ e^{2Rt }\chi_0 R \int_{s = 0}^t  e^{-2Rs} \left(|\nabla S| \left(\rho-K\right)_+ \varphi\right)(T+s,x-(t -s)v) \, ds\, . \\
\end{split}
\end{equation*}
As $\eps\to 0$ we obtain the following integral inequality,
\begin{equation*}
\begin{split}
&(f\varphi-k\varphi)_+(T+t ,x,v) \leq e^{2Rt }(f\varphi-k\varphi)_+(T,x-t  v,v)\\ 
&\qquad+ e^{2Rt }\chi_0 R \int_{s = 0}^t  e^{-2Rs} \left(|\nabla S| \left(\rho\varphi-K\varphi\right)_+\right)(T+s,x-(t -s)v) \, ds \, .
\end{split}
\end{equation*}

Next compute similarly as in the Appendix (by the help of the dispersion estimate Lemma \ref{lem:dispersion}) for admissible exponents $(p,q)$ (Definition \ref{def:admissible}),
\begin{eqnarray*}
\|(f\varphi-k\varphi)_+(T+t ,x,v)\|_{L^p_xL^q_v}& \leq& e^{2Rt } \|(f\varphi-k\varphi)_+(T,x-t  v,v)\|_{L^p_xL^q_v}  \\
&& \ \  + \chi_0 C(V) e^{2Rt } \int_{s=0}^t  e^{-2Rs} \dfrac1{(t -s)^{2(1/q-1/p)}}   \||\nabla S| (\rho\varphi-K\varphi)_+(T+s)\|_{L^q} ds \\
& \leq& e^{2Rt } \|(f\varphi-k\varphi)_+(T,x-t  v,v)\|_{L^p_xL^q_v}  \\
&& \ \  + \chi_0 C(V) e^{2Rt } \int_{s=0}^t  e^{-2Rs} \dfrac1{(t -s)^{2(1/q-1/p)}}   \|\nabla S(T+s)\|_{L^{p^\star}} \|(\rho\varphi-K\varphi)_+(T+s)\|_{L^p} ds \\
&  \leq & e^{2Rt } \|(f\varphi-k\varphi)_+(T,x-t  v,v))\|_{L^p_xL^q_v}  \\
&&\  \  + \chi_0 C(V) e^{2Rt } \int_{s=0}^t  e^{-2Rs}\dfrac1{(t -s)^{2(1/q-1/p)}}  \| \rho(T+s)\|_{L^{r}} 
\|(\rho\varphi-K\varphi)_+(T+s)\|_{L^p} ds\, ,
\end{eqnarray*}
where the H\"older exponents are given by $1/p^\star + 1/p = 1/q$ and
the Sobolev exponent ($r<2$) is given by $1/r = 1/2 + 1/p^\star = 1/2 + 1/q - 1/p < 1$.

Furthermore, because the positive part $(\cdot)_+$ is a convex function, and $V$ is a bounded set, we have 
by Jensen's inequality,
\begin{align*}& (\rho\varphi  - K\varphi)_+(T+s,x)  = \left(\int_{v\in V} |V| (f\varphi-k\varphi)(T+s,x,v) 
\,\dfrac{ dv}{|V|}\right)_+ \leq \int_{v\in V} \left(|V|(f\varphi-k\varphi)\right)_+(T+s,x,v) \dfrac{ 
dv}{|V|}\, , \\
& \|(\rho\varphi-K\varphi)_+(T+s)\|_{L^p} \leq  \|(f\varphi-k\varphi)_+(T+s)\|_{L^p_xL^1_v}\leq 
C(q,V)\|(f\varphi-k\varphi)_+(T+s)\|_{L^p_xL^q_v}\, .
\end{align*}

As soon as $f(T,x,v)$ remains below $k(x,v)$, Lemma \ref{lem:LpLq comparison} guarantees that the free 
transport contribution $f(T,x-tv,v)$ belongs to $L^p_xL^q_v$ for admissible exponents $(p,q)$. As a 
consequence of the local in time existence result of Proposition \ref{ref:local existence bis} (in the 
Appendix), we can ensure that $f(T+t,x,v)$ belongs to $L^p_xL^q_v$ for small time $t>0$. Thus $\rho(T+t,x)$ 
belongs to $L^r$ for any $r<2$ in particular.

Therefore the Gronwall lemma guarantees that $\|(f\varphi - k\varphi)_+\|_{L^p_xL^q_v}$, if it is zero up 
to some time $T$, it remains zero for small later times $T+t $.

We have proven a comparison principle which prevents solutions to  \eqref{kinmodel} from blowing-up.
\end{proof}

\thanks{{\em Acknowledgements.} The authors are grateful to Beno\^it Perthame for having addressed this 
challenging issue, and for fruitful comments on this work. VC thanks Thibaut Allemand for stimulating 
discussions. NB would like to thank the Laboratoire Jacques-Louis Lions of the Universit\'e Pierre et Marie Curie
and the D\'epartement de Math\'ematiques et Applications of the \'Ecole Normale Sup\'erieure for their 
hospitality and financial support during his sabbatical leave in the spring semester of 2008. }

\newpage

\section*{Appendix A: Solutions having spherical symmetry}

The notion of spherical symmetry in kinetic theory is contained in the following Definition \ref{def:spherical sym}. Recall that the set of admissible velocities is the ball $V = \mathcal B(0,R)$.
\begin{definition} \label{def:spherical sym}
A function $f(x,v)$ defined for $(x,v)\in\R^2\times V$ is spherically symmetric  if for 
every rotation $\Theta$ of $\R^2$ we have $f(\Theta x, \Theta v) = f(x,v)$.
\end{definition}

If $f(x,v)$ is spherically symmetric then the space density $\rho(x)=\int_{v\in V} f(x,v) dv$
is spherically symmetric in the usual sense, i.e $\rho(\Theta x)=\rho(x)$ for all rotations $\Theta$. Therefore $\rho$ depends only on $r=|x|$. Abusing notations we write both $\rho(x)$ and $\rho(r)$
but the meaning will always be clear from the context. 

If $f(x,v)$ is spherically symmetric then its current $j(x)=\int_{V} v f(x,v) dv$ does not necessarily point
in the direction of $x$. For example, if 
$ f(x,v)=x\cdot v^{\perp}=x_2 v_1 - x_1 v_2\, ,
$ then, 
\begin{equation*}
 j(x)=-\frac{\pi R^4}{4}(-x_2 , x_1)\, \perp\, x\, .
\end{equation*}
However, if we decompose
\begin{equation}\label{j1}
 j(x)= j^{\parallel}(x)\frac{x}{|x|} + 
j^\perp(x)\frac{x^{\perp}}{|x|}\, ,
\end{equation}
then we have for every rotation $\Theta$,
\begin{equation*}\label{j2}
 j(\Theta x) = \int_{ V} vf(\Theta x,v)  \, dv = \int_{  V} (\Theta w)f(\Theta x, \Theta w) \,  d(\Theta w)\Theta j(x) \, .
\end{equation*}
Therefore the decomposition's coefficients in \eqref{j1} are both spherically symmetric:
\begin{equation*}\label{j3}
j^\parallel(\Theta x) = j^\parallel(x)\, , \quad \mbox{and} \quad
j^\perp(\Theta x) = j^\perp(x)\, .
\end{equation*}
Abusing notation we write
\begin{equation*}\label{j4}
 j(x)= j^{\parallel}(r)\frac{x}{r} + j^{\perp}(r) \frac{x^{\perp}}{r}\, .
\end{equation*}
We can then simply derive the following identities which will be crucially used in Section \ref{sec:BU}:
\begin{align}
&x \cdot j(x)= r j^{\parallel}(r)\label{j5}\, ,\\
&\nabla \cdot j(x) = \frac{1}{r}\left(r j^{\parallel}(r) \right)'\, .\label{j6}
\end{align}

\subsection*{The kinetic system \eqref{kinmodel} preserves spherical symmetry}

If we start with   spherically symmetric initial data $f_0(x,v)$ then Proposition 
\ref{localexthm} guarantees 
the existence of a local in time solution $f(t,x,v)$. It is not difficult to
verify that for any rotation $\Theta$ the function $f(t,\Theta x, \Theta v)$ is also a solution to 
\eqref{kinmodel} (see below). Therefore, by the uniqueness part of Proposition 
\ref{localexthm} they have to coincide. It follows that
$f(t,x,v)$ is spherically symmetric throughout the time interval of existence. 

Let $(f,S)$ be a solution and $\Theta $ be a rotation of $\R^2$. Define $(g,Q)$ by 
\begin{equation*}\label{r1}
 g(t,x,v)=f(t,\Theta x, \Theta v)\, ,\quad  Q(t,x)=S(t,\Theta x)\, .
\end{equation*}
On the one hand,
\begin{eqnarray*}
\partial_t g(t,x,v) + v \cdot \nablax g (t,x,v) & = & 
\partial_t f (t,\Theta x, \Theta v) + v \cdot \Theta^{T}(\nablax f)(t, \Theta x, \Theta v) \\
&=&  \partial_t f (t,\Theta x, \Theta v) + \Theta v \cdot (\nablax f)(t, \Theta x, \Theta v)\, .
\end{eqnarray*}
On the other hand,
\begin{align*}
&\int_{V}T[S](t, \Theta x, \Theta v, v') f(t, \Theta x, v')\, dv' 
- \int_{V}T[S](t, \Theta x, v', \Theta v) f(t, \Theta x, \Theta v)\, dv' \\
&\qquad =\int_{V}\left(\Theta v \cdot (\nabla S)(t,\Theta x) \right)_{+} f(t, \Theta x, v')\, dv' 
- \int_{V}\left( v' \cdot (\nabla S)(t,\Theta x) \right)_{+} f(t, \Theta x, \Theta v)\, dv' \\
&\qquad =\int_{V}\left( v \cdot \Theta^{T}(\nabla S)(t,\Theta x) \right)_{+} f(t, \Theta x, \Theta w)\, dw 
- \int_{V}\left( w \cdot \Theta^{T}(\nabla S)(t,\Theta x) \right)_{+} f(t, \Theta x, \Theta v)\, dw\\
&\qquad=\int_{V}\left( v \cdot (\nabla Q)(t, x) \right)_{+} g(t,  x,  w)\, dw 
- \int_{V}\left( w \cdot (\nabla Q)(t,\Theta x) \right)_{+} g(t,  x,  v)\, dw\\
&\qquad=\int_{V}T[Q](t,  x,  v, w) g(t,  x, w)\, dw
 - \int_{V}T[Q](t, \Theta x, w,  v) g(t,  x,  v)\, dw .
\end{align*}
Also
\begin{equation*}
 -\Delta Q(x) + \alpha Q(x) =- (\Delta S) (\Theta x) + \alpha S(\theta x) 
=\int_{V} f(\Theta x, v) dv  =\int_{V} f(\Theta x, \Theta w ) dw 
=\int_{V} g( x,  w ) dw\,  .
\end{equation*}


\section*{Appendix B: Existence and uniqueness in a weaker framework}

The goal of this appendix is to provide a variant of Proposition \ref{localexthm} in a framework well-suited for proving the global existence result of Section \ref{sec:existence}. As a matter of fact, the reference function 
$k(x,v)$
used there  belongs to $L^p_{\mathrm{loc},x}L^q_v$ for any $1\leq p < 2/\gamma$ and $1\leq q < 1/\gamma$. On the other hand,  Proposition \ref{localexthm} deals with solutions lying in $L_x^pL_v^p$, $p>2$, thus it does not cover properly the case $\gamma>1/2$. 

\begin{assumption}[Initial datum, precised version]
\label{ass:f0 bis}
Assume that the initial density $f_0(x,v)\geq 0$ belongs to $L^1_{x,v}$ and satisfies the estimate $\|f_0(x-t v,v)\|_{L^p_xL^q_v}\leq \bC_0$ for small times $t\geq 0$, and for some couple of exponents $(p,q)$ verifying: 
\begin{equation}
\label{eq:precise exponents}
2< p \,, \quad 1< q  \, , \quad 0\leq\dfrac1q - \dfrac1p < \dfrac12\, .
\end{equation}
\end{assumption}

Observe that Assumption A\ref{ass:f0 bis} is satisfied as soon as $f_0\in L^p_{x,v}$ for some $p>2$. Thus it is weaker than Assumption A\ref{ass:f0}, except for the condition of spherical symmetry.

\begin{proposition}
\label{ref:local existence bis}
Assume that the initial density $f_0$ verifies Assumption A\ref{ass:f0 bis} for some couple of exponents $(p,q)$ such that \eqref{eq:precise exponents} holds true. Then there exists a unique (local in time) solution to system \eqref{kinmodel} with $f(t,x,v)\in L^p_xL^q_v$.
\end{proposition}

Before the proof of this Proposition, let start with a general feature of kinetic transport equations with source and decay terms.
The solution of $\partial_t h + v\cdot \nablax h + \lambda(t,x,v) h = g$
with vanishing initial data is given by the Duhamel's representation,
\[
h(t,x,v)=\int_{s=0}^{t} g(s,x-(t-s)v,v) \exp\left\{-\int_{\tau = s}^{t} \lambda(\tau,x-(t-\tau)v,v)\,  d\tau\right\}\,  ds\,  .
\]
In case  $\lambda$ is a nonnegative function, we obtain, 
\begin{eqnarray*}
\abs{h(t,x,v)} &\leq &\int_{s=0}^{t} \abs{g(s,x-(t-s)v,v)} 
\exp\left\{-\int_{\tau=s}^{t} \lambda(\tau,x-(t-\tau)v,v)\, d\tau\right\}\, ds\,\\
& \leq &  \int_{s=0}^{t} \abs{g(s,x-(t-s)v,v)}\,  ds\, .
\end{eqnarray*}

\begin{proof}
We aim to write directly a fixed-point argument under 
the reference norm,
\begin{equation}\label{norm bis}
  \norm{g}{Y}  =\sup_{0\leq t \leq T} \left(\norm{g(t,x,v)}{L^{1}_{x,v}} + 
 \norm{g(t,x,v)}{L^{p}_{x}L^q_v}\right)\, .
\end{equation}
We start from 
\begin{multline}\label{newappendixB:1}
 \partial_t (f_1-f_2) + v\cdot \nabla_x (f_1-f_2) + \chi_0\omega \abs{\nabla S_2}\left( f_1 - f_2\right)=\\
\chi_0 \left(\left(v \cdot \nabla S_1\right)_{+} - \left(v \cdot \nabla S_2\right)_{+}\right)\rho_1 
 + \chi_0 \left(v \cdot \nabla S_2\right)_{+} \left( \rho_1 - \rho_2\right)
   - \chi_0\omega\left(\abs{\nabla S_1} - \abs{\nabla S_2} \right)f_1\, . 
\end{multline}
Applying the preliminary observation to equation
\eqref{newappendixB:1} we obtain
\begin{eqnarray*}
|f_1-f_2|(t) & \leq &  C  \int_{s = 0}^t  
 \Big( \left\|\nabla S_1 - \nabla S_2\right\|_\infty  \abs{\rho_1 (s,x-(t-s)v)}
 + \|\nabla S_2\|_\infty \left| \rho_1 - \rho_2\right|(s,x-(t-s)v) \Big)\, ds \\
 &&\qquad  + C  \int_{s = 0}^t  
   \left\| \nabla S_1  -  \nabla S_2 \right\|_\infty \abs{f_1(s,x-(t-s)v,v)}\, ds\, ,   
\end{eqnarray*}
Therefore we are able to develop a dispersion technique as usual,
\begin{eqnarray*}
\|(f_1-f_2)(t)\|_{L_x^pL_v^q} & \leq &  C  \int_{s = 0}^t  (t-s)^{-2(1/q-1/p)} 
 \Big( \left\|(\nabla S_1 - \nabla S_2)(s)\right\|_\infty \|\rho_1 (s)\|_{L^q}
 + \|\nabla S_2\|_\infty \| (\rho_1 - \rho_2)(s)\|_{L^q} \Big)\, ds \\
 &&\qquad  + C  \int_{s = 0}^t  
   \left\| (\nabla S_1  -  \nabla S_2)(s) \right\|_\infty \| f_1(s,x-(t-s)v,v)\|_{L^pL^q}\, ds
\\
& \leq &  C  \left( \|f_1\|_Y + \|f_2\|_Y \right) 
\left(  \int_{s = 0}^t   (t-s)^{-2(1/q-1/p)} \, ds\right) \| f_1-f_2\|_Y\\
 &&\qquad  +  C \left(  \int_{s = 0}^t  
   \| f_1(s,x-(t-s)v,v)\|_{L^pL^q}\, ds\right)\| f_1-f_2\|_Y \, ,   
\end{eqnarray*}
where we have used Lemma \ref{lem:elliptic}.
In parallel, we get a bound for $\| f_1(s,x-(t-s)v,v)\|_{L^p_xL^q_v}$. We argue as follows: $f_1$ solves
\[
\partial_t f_1 + v\cdot \nabla_x f_1 + \chi_0 \omega \abs{\nabla S_1} f_1 = \chi_0 \left(v \cdot \nabla S_1\right)_{+}  \rho_1\, ,
\]
and since the coefficient $\chi_0 \omega \abs{\nabla S_1}$ is non-negative, we can argue as above to get
\[
\abs{f_1(s,x,v)}\leq \abs{f_1(0,x-sv,v)} + C \int_{\tau = 0}^s \|\nabla S_1(\tau)\|_\infty 
\abs{\rho_1(\tau,x-(s-\tau)v,v)}\, d\tau\, ,
\]
and eventually
\begin{align*}
 \abs{f_1(s,x-(t-s)v,v)}  & \leq   \abs{f_1(0,x-tv,v)} + C \int_{\tau = 0}^s \|\nabla S_1(\tau)\|_\infty 
\abs{\rho_1(\tau,x-(t-\tau)v,v)}\, d\tau\, ,\\
 \|f_1(s,x -(t-s) v,v)\|_{L^p_xL^p_q}  & \leq  \|f_0(x-tv,v)\|_{L^p_xL^q_v} + C \int_{\tau=0}^s (t-\tau)^{-2(1/q-1/p)} \|\nabla S_1(\tau)\|_{L^\infty}\|\rho_1(\tau)\|_{L^q}\, d\tau \\
&  \leq  \bC_0 + C \|f_1\|_Y^2 \int_{\tau=0}^t (t-\tau)^{-2(1/q-1/p)} \, d\tau
\,  ,
\end{align*}
from which we deduce that the flow is contractant for small time w.r.t. the reference norm $\|\cdot\|_Y$, and relatively to the initial datum.

\end{proof}

\end{document}